\documentclass[11pt,a4paper]{article}
\usepackage[margin=25mm,headheight=14pt,headsep=7mm,
footskip=11mm]{geometry}
\usepackage[T1]{fontenc}
\usepackage{libertinus-type1}
\usepackage{amsmath,amssymb,amsthm,mathtools,mathrsfs,bm}
\usepackage{microtype,enumitem,booktabs,array}
\usepackage{tikz,tikz-cd}
\usetikzlibrary{arrows.meta,positioning,calc}
\usepackage[hidelinks,unicode]{hyperref}
\hypersetup{pdftitle={Symplectic singularity as a hyperKahler cone},pdfsubject={Reference-radius comparison replaces metric-radius continuity, 15 October 2026}}
\numberwithin{equation}{section}
\newtheorem{theorem}{Theorem}[section]
\newtheorem{proposition}[theorem]{Proposition}
\newtheorem{lemma}[theorem]{Lemma}
\newtheorem{corollary}[theorem]{Corollary}
\theoremstyle{definition}

\theoremstyle{remark}
\newtheorem{remark}[theorem]{Remark}
\newtheorem*{remark*}{Remark}
\newcommand{\C}{\mathbb C}\newcommand{\R}{\mathbb R}\newcommand{\Z}{\mathbb Z}\newcommand{\Q}{\mathbb Q}
\newcommand{\A}{\mathbb A}\newcommand{\Gm}{\mathbb G_m}\newcommand{\OO}{\mathcal O}\newcommand{\m}{\mathfrak m}
\newcommand{\cW}{\mathcal W}\newcommand{\cV}{\mathcal V}\newcommand{\cR}{\mathcal R}\newcommand{\cB}{\mathcal B}\newcommand{\cD}{\mathcal D}\newcommand{\cU}{\mathcal U}
\newcommand{\RV}{\widehat{\mathcal R}_{\mathcal V}}\newcommand{\hR}{\widehat{\mathcal R}}
\newcommand{\cten}{\widehat\otimes}\newcommand{\reg}{\mathrm{reg}}\newcommand{\test}{\mathrm{test}}\newcommand{\Pois}{\mathrm{Pois}}
\newcommand{\qmax}{q_{\max}}
\newcommand{\HP}{\operatorname{HP}}\newcommand{\DP}{\operatorname{DP}}\newcommand{\ZP}{Z_{\mathrm{Pois}}}

\DeclareMathOperator{\Spec}{Spec}\DeclareMathOperator{\Hom}{Hom}
\DeclareMathOperator{\Der}{Der}\DeclareMathOperator{\Sym}{Sym}
\DeclareMathOperator{\im}{Im}\DeclareMathOperator{\id}{id}\DeclareMathOperator{\wt}{wt}\DeclareMathOperator{\gr}{gr}
\DeclareMathOperator{\Ann}{Ann}\DeclareMathOperator{\pr}{pr}\DeclareMathOperator{\rk}{rk}\DeclareMathOperator{\codim}{codim}\DeclareMathOperator{\depth}{depth}
\newcommand{\drel}{d_{\mathrm{rel}}}
\setlist[enumerate]{leftmargin=2em,itemsep=3pt,topsep=4pt}
\allowdisplaybreaks[2]\clubpenalty=10000\widowpenalty=10000\displaywidowpenalty=10000
\title{Symplectic singularity as hyperK\"ahler cone}
\author{Yuji Odaka}\date{\today}

\begin{document}

\maketitle
\begin{abstract}
We prove that any symplectic singularity has a structure of hyperK\"ahler cone. In particular, it is equivalent to the vertex of a conical symplectic variety, and it is automatically K-polystable.
\end{abstract}

\tableofcontents

\section{Introduction}\label{sec:intro}
Around two decades ago, Kaledin (partially) predicted that any symplectic singularity
(\cite{Beauville})
should be a conical symplectic singularity in the sense it ``locally admits'' $\C^*$-
action (cf., e.g., \cite{Kaledin00, Kaledin}). 

In a fairly different context of canonical K\"ahler
metrics, much later, Donaldson-Sun's theory of metric tangent cones for singular 
K\"ahler-Einstein metrics around more general log terminal singularities revealed its algebraic nature via the 
underlying two step process 
\begin{equation}\label{eq:conification}
(x\in X)\rightsquigarrow(0\in W)\rightsquigarrow(0\in C),
\end{equation}
of algebraic nature. Here $(0\in W)$ and $(0\in C)$ are 
pointed 
affine varieties with good action of 
canonical algebraic tori, with only fixed points at the vertex $0$. More precisely, they are a 
(K-semistable) log terminal cone and a (K-polystable) log terminal 
cone \cite[\S3]{DS} respectively, 
in the sense of 
\cite{CollinsSzek, CollinsSzek2}.
\footnote{In their works, these are called Fano cones. Indeed, 
affine cones for (pluri-)anticanonical line bundles over Fano varieties are fundamental examples, with regular Reeb vector fields.}
After these works and the volume-minimization 
phenomena in Sasakian geometry by Martelli-Sparks-Yau  
\cite{MSY1, MSY2}, 
C.Li introduced an algebraic framework for this 
construction which further developed; 
see \cite{LiNV,LiXu,LiWangXu,XZ} (or a survey \cite{LLX}). 

Our earlier paper \cite{NO} with Y.~Namikawa first found an unexpected link between these two stories for symplectic 
singularity $x\in X$, to identify 
the germs of $x\in X$, $0\in W$, $0\in C$, under mild hypotheses. 
Note that such identifications provide ``canonical" local algebraic torus action 
around the germ of $x\in X$ with geometric meanings. 
The proof combines algebraic geometry, analysis of hyperK\"ahler 
metrics and Poisson deformations. 
Later, Xu-Zhuang \cite{XZ} subsequently removed the hypothesis for the part of  
comparison between $X$ and $W$, generalizing \cite[\S 4]{NO}, 
depending on certain singular Hermitian metric 
constructed by Guenancia (in their appendix, after \cite{DGP}). 
The question addressed in this paper is the rigidity of the second arrow: 
proving $W\simeq C$ (Theorem \ref{thm:rigidity}), removing the hypothesis in 
\cite[\S5]{NO}. As a consequence, any symplectic singularity germ  is 
equivalent to K-polystable conical symplectic variety with hyperK\"ahler cone 
metric. 

For these new approaches, we first 
put both processes in \eqref{eq:conification} into degenerating
families; $X\rightsquigarrow W$ as isotrivial degeneration 
over an affine toric variety 
(\cite{LiXu}, \cite[\S 2]{Od24}, 
\cite[\S 2]{NO}) and 
$W\rightsquigarrow C$ in \cite[\S 3]{DS}
and \cite[Theorem 1.2]{LiWangXu}
as a test configuration
\begin{equation}\label{Wtc}
p\colon \mathcal{W}\to \mathbb{A}^1,
\end{equation}
with the general fiber $0\in W$ and central fiber $0\in C$.
Then, the main idea proposed in the program
\cite{NO} is that, one can analyze
these two families via
Poisson deformation theory (cf., e.g., \cite{NamikawaAffine})
as far as we can transplant algebraic symplectic structure of $X$ to
$W$ and $C$. 

\subsection*{Notations and setup}\label{sec:conventions}
In this paper, varieties are all assumed to be integral and of finite type over $\C$. We write $n=\dim X=\dim W=\dim C=2m$ for their common complex dimension. For a normal variety $Y$, put $\Omega_Y^{[p]}=j_*\Omega_{Y^{\reg}}^p$ where $j:Y^{\reg}\hookrightarrow Y$ is the open immersion from the regular locus $Y^{\reg}$ of $Y$. A symplectic variety carries a closed non-degenerate 2-form on its regular locus extending regularly to a resolution \cite{Beauville}. Its singularities are canonical and hence klt. A finite map is quasi-\'etale if it is \'etale outside a subset of the target of codimension at least two.

For an algebraic torus $T$, write $M$ for its character lattice. A Reeb vector $\xi\in\Hom(\Gm,T)\otimes\R$ is positive on every nonconstant function character. Its holomorphic Euler field is $E=Z_\xi$; thus $Ef=\langle\gamma,\xi\rangle f$ on a function of character $\gamma$. The real radial field is $R=E+\bar E$, and the real Reeb field is $IR$. A good integral action has only positive nonconstant weights and degree-zero ring $\C$.
Write $T_\xi\subset T$ for the algebraic torus generated by $\xi$; evaluation $\chi\mapsto\langle\chi,\xi\rangle$ is injective on its character lattice. The standard application uses $T=T_\xi$, while the hypotheses below allow a larger torus. Put $N=\Hom(\Gm,T)$ and $N_\R=N\otimes_\Z\R$. 

Now, we discuss our main setup. See \cite{NO} 
with its preparation \S 2 for the details. 
As a priori generalization of what appears as \eqref{Wtc},
let $p:\cW=\Spec\cR\to\A^1_t$ be normal and flat, with normal central fiber $C$, and a fixed punctured identification $\cW|_{t\ne0}\simeq W\times(\A^1_t\setminus0)$, and henceforth assume the following. (Our main 
applications, these are satisfied, as it is explained.) 
 
\begin{enumerate}[label=\textup{(H\arabic*)}]
\item\label{hyp:grading} A fiberwise torus $T$ fixes $t$. A commuting auxiliary 
test torus $G_{\test}\simeq\Gm$ 
acts on the base $\A^1_t$ 
of $\mathcal{W}$ with the weight $\wt(t)=d>0$ 
(often $d=1$; but we allow general $d$ for a technical reason, to lift the actions on 
a certain finite cover of $C$ later.) 
The same $\xi$ has positive weights on $\Gamma(W,\OO_W)$ and $\Gamma(C,\OO_C)$; $\cR^T=\C[t]$ and, 
for each $\gamma \in M$, the corresponding $T$-eigen sub $\C[t]$-module $\cR_\gamma$ is finite free $\C[t]$-module 
(cf., \cite[2.12, 2.13]{OdCompact} and 
\cite{CollinsSzek, CollinsSzek2}). 
\item\label{hyp:form} $W$ has a $T$-homogeneous symplectic form $\sigma_W$ of character $\tau$, where $\langle\tau,\xi\rangle=2$. In particular, 
it is a symplectic variety. 
\item\label{hyp:metric}
$C$ is a K-polystable log terminal cone of dimension $n=2m$,
normalized by $A_C(\xi)=2m$, with a weak Calabi-Yau cone
metric $g$ of Reeb vector $\xi$.
Writing $G_{\test,c}\simeq S^1$ for the maximal compact subgroup
of $G_{\test}$, we choose $g$ invariant under the induced action
of $T_c\times G_{\test,c}$ on $C$.
\end{enumerate}
Here $A_C(\xi)$ is the $\xi$-weight of a nowhere-vanishing pluricanonical frame of $\OO(\ell K_C)$ divided by its index $\ell$. As pointed out by \cite{LiNV}, 
it coincides with the log discrepancy for the associated 
valuation of $C$ to $\xi$. 
By a weak cone metric, we mean that there is a positive smooth radius $r$ on $U=C^{\reg}\setminus\{o\}$, homogeneous of degree one under the real dilations generated by $R$ and invariant under $T_c\times G_{\test,c}$. Its square extends to a nonnegative locally bounded plurisubharmonic function on $C$. We require a continuous proper reference radius $r_\xi:C\to\R_{\ge0}$ with the same homogeneity, positive off the vertex and zero there, and a constant $\kappa\ge1$ such that $\kappa^{-1}r_\xi\le r\le\kappa r_\xi$ on $U$. On $U$ the metric satisfies
\begin{equation}\label{eq:coneMA}
g=dr^2+r^2g_L,\qquad \omega=dd^c(r^2/2),\qquad \omega^n=c_0dV_{\mathrm{can}},\quad c_0>0.
\end{equation}
The last equality holds as bounded-potential Monge-Amp\`ere measures (cf.\ \cite{Bedford82}). On a log resolution $\pi:\widetilde C\to C$ the canonical measure has the expression
\begin{equation}\label{eq:canmeasure}
\pi^*dV_{\mathrm{can}}=b\prod_i|s_i|_{H_i}^{2a_i}dV_H,\qquad
K_{\widetilde C}=\pi^*K_C+\sum_i a_iD_i,\quad a_i>-1,
\end{equation} 
with $b>0$ smooth. 
(Although we work with general klt cones
in this section, the cone $C$ arising from a symplectic
singularity is canonical Gorenstein by
\cite[Corollary 4.3]{NO}. Thus, for our main application,
we have $a_i \ge 0$ for all $i$.) 
The metric is smooth and Ricci-flat on $C^{\reg}$.
Metric existence and regularity are established by 
\cite[Theorem 2.9]{LiSoliton, JubertPan} 
and \cite[Theorem 1.5]{Nghiem}; Lemma \ref{lem:radius-comparison} records 
the bounded relative potential and the comparison with $r_\xi$ in 
\cite[\S 4.1]{NO}. Applying the construction to the effective 
torus generated by $T$ and $G_{\test}$ gives the required compact-torus
invariance without changing the fiberwise action. 

\vspace{3mm}

For our application to symplectic singularities, 
we are in the above setup. Indeed, in addition to the 
Donaldson-Sun 
$2$-step degenerations $(x\in X)\rightsquigarrow 
W\rightsquigarrow C$ \eqref{eq:conification} 
(cf., e.g., 
\cite[\S 2]{NO} to review the details) 
\cite[\S 4]{NO}, \cite[6.2]{XZ} supply 
the symplectic form $\sigma_W$, 
our \S \ref{sec:intermediate} in this paper shows its 
homogeneity with weight $2$. 
These establish \ref{hyp:grading}-\ref{hyp:metric}. 

\begin{theorem}[Bochner principle]\label{thm:bochner}
Let $(C,\xi)$ be a K-polystable log terminal 
cone of complex 
dimension $n\ge 2$, normalized by $A_C(\xi)=n$, and let $g$
be a weak Calabi-Yau cone metric as in \ref{hyp:metric}.
Here the weak metric is understood in the sense specified above,
including the locally bounded psh extension of $r^2$,
the uniform comparison with the proper reference radius $r_\xi$,
and \eqref{eq:coneMA}.

For $1\le p\le n$, let $\alpha\in H^0(C,\Omega_C^{[p]})$
satisfy
\[
  \delta_s^*\alpha=s^p\alpha\qquad(s>0),
\]
where $\delta_s$ denotes the real dilation generated by $R$.
Then $\nabla^g\alpha=0$ on $C^{\reg}$.
\end{theorem}

The following theorem is our main theorem. Under  \ref{hyp:grading}-\ref{hyp:metric}, establishes 
the $2$nd-step rigidity predicted in \cite[\S5]{NO} (see also Conjecture 1.3 
in {\it loc.cit}, which is an even stronger statement, due to Theorem 
1.4 and its proof). 
\begin{theorem}[Rigidity of the second degeneration]\label{thm:rigidity}
Assume that $p:\cW\to\A^1_t$ satisfies
\ref{hyp:grading}-\ref{hyp:metric}.
Then there is a $T\times G_{\test}$-equivariant isomorphism 
\[
\cW\simeq C\times\A^1_t
\]
over $\A^1$, inducing the identity on the marked central fiber. On the product the test torus has its induced action on $C$ and weight $d$ on $t$. In particular $W\simeq C$ as $T$-varieties.
\end{theorem}
\begin{corollary}[Symplectic singularity as hyperK\"ahler cone]\label{cor:hk}
For any symplectic singularity $x\in X$, the associated
K-polystable affine
cone $0\in C$ is a conical symplectic variety with the same analytic 
germ as $x\in X$. 
The metric $g$ admits parallel complex structures $I,J,K$ making
$C^{\reg}$ a hyperK\"ahler cone. 
In particular, for any symplectic singularity germ, 
it admits a hyperK\"ahler metric on its smooth locus. 
\end{corollary}

This generalizes our original \cite[Theorem 1.1, 1.4]{NO}, 
which proved the same under mild hypotheses. 
In {\it loc.cit} \S 6.4 (Theorem 6.16), 
we proved under additional 
hypotheses that the dilating $\mathbb{R}_{>0}$-action 
extends 
to a {\it non-holomorphic} action of 
$\mathbb H^*$ on the regular locus for the quarternion $\mathbb{H}$. In particular,
the Reeb flow is periodic, so the cone is quasi-regular
and $T_\xi\simeq\mathbb C^*$. 
In {\it loc.cit}, we also 
conjectured that the same holds for any symplectic 
singularity 
(\cite[Conjecture 4.11]{NO}). 
Note that their quotients 
are then {\it possibly singular} log K-polystable (singular K\"ahler-Einstein) contact 
log Fano pairs as also explained in {\it loc.cit} 
(cf., \cite[Remark 1.6, \S 6.4]{NO}). 

\subsection{The main technical ingredients}\label{sec:inputs}
Our paper has three components of different natures. 
The first is on Theorem \ref{thm:bochner} and the other two are on 
Theorem \ref{thm:rigidity}. 

\begin{enumerate}
\item Analytic preparation (Theorem \ref{thm:bochner}): a Bochner type argument on {\it singular} $C$ with 
appropriate cutoffs proves Theorem \ref{thm:bochner} 
(when $0\in C$ is isolated singularity, it is straightforward.) 

\item Algebraic (non-geometric) preparation: use Harrison-Poisson cohomology, which captures 
Poisson deformations, and its 
K\"unneth type formulas (only degree $2, 3$ case are needed here) of
\cite[Lemma A.5]{GK}. Extending the main observation 
of the triviality of scale up Poisson deformations in 
\cite[\S 3]{NO}, using similar ideas, we show the formal rigidity of 
certain product type Poisson algebras. This rigidity applies to 
somewhat more general Poisson geometric situation than our conical symplectic 
varieties. 

\item Using two preparations of different natures as above, 
we show that, a finite quasi \'etale Galois covering $V$ of 
$C$ has a product structure. Then, using the above result in general Poisson 
deformation, we show formal Poisson rigidity of $V$ and $C$. This leads to 
the triviality of $\mathcal{W}$, proving Theorem \ref{thm:rigidity}. 

\end{enumerate}

\subsection{Organization of the proof}\label{sec:organization}
Our main theorem is a rigidity statement (Theorem \ref{thm:rigidity}), which applies to any family satisfying
\ref{hyp:grading}-\ref{hyp:metric}, not restricted to those 
arises from applying Donaldson-Sun theory to a symplectic 
singularity germ. Corollary \ref{cor:hk} simply follows from Theorem \ref{thm:rigidity}. 
Sections \S \ref{sec:analysis} and \S \ref{sec:PD} are 
preparatory in 
nature; Section \S \ref{sec:blocks} contains the core 
argument for Theorem \ref{thm:rigidity}. 
These two sections are independent 
analytic and algebraic preparations, respectively. 

Section \S \ref{sec:analysis}  begins by 
recalling the basics of the weak Calabi-Yau metrics on $C$. 
Then, after a basic analysis on reference radius functions in
\S \ref{sec:reference-radius}, for the fixed Reeb vector, we 
prove a Bochner principle type result for $C$ (Theorem 
\ref{thm:bochner}). Note that it a priori can have non-isolated 
singularities; so that we need a delicate analysis on 
cut-off functions. 

On the other hand, \S \ref{sec:PD} stufies Poisson 
deformations of products. Firstly, in \S 
\ref{sec:PD-cohomology}-\ref{sec:PD-kunneth}, we recall Poisson 
cohomology and the K\"unneth type formula from \cite{GK} 
and also put some additional explanations. 

The core \S \ref{sec:blocks} relies on both \S \ref{sec:analysis} and \S \ref{sec:PD}. The first preliminary
\S \ref{sec:intermediate} verifies the hypothesis for the weight \ref{hyp:form} for the application to 
symplectic singularities; 
all subsequent constructions only assume and use 
the stated hypotheses. \S \ref{sec:flag}-\ref{sec:cover} 
split $T_{C^{\rm reg}}$ into foliations by 
applying the elementary divisor theory to $\mathcal{W}$ and 
construct finite Galois quasi-\'etale 
cover of $C$ as a product of 
conical symplectic varieties. 
Sections \ref{sec:extendcover}-\ref{sec:tensors} continue with this same
cover, extending it over the formal family and proving Casimir vanishing. 
Only then, in \S \ref{sec:recursive}, is the splitting theorem of 
\S \ref{sec:PD} applied to prove the desired rigidity 
in a somewhat formal weaker sense. 
\S \ref{sec:conclusion} remains a separate final step: it  
descends 
and algebraizes the formal product, then show the same for 
the original germ. 

The following flowchart gives a quick overview of our paper. The middle column follows the core argument, the side 
columns supply the preparatory sections. 

\begin{center}
\begin{tikzpicture}[
  flowbox/.style={draw, rounded corners=1pt, line width=.4pt,
    align=center, text width=4.3cm, minimum height=1.05cm,
    inner xsep=2mm, inner ysep=1.4mm, font=\small},
  inputbox/.style={flowbox,dashed},
  flowhead/.style={align=center,text width=4.3cm,font=\small\bfseries},
  flowarrow/.style={-{Latex[length=1.7mm,width=1.2mm]},line width=.45pt}
]
\node[flowhead] at (0,1.1) {\S\ref{sec:analysis}\\Analytic preparation};
\node[flowhead] at (5.35,1.1) {\S\ref{sec:blocks}\\Core argument};
\node[flowhead] at (10.7,1.1) {\S\ref{sec:PD}\\Algebraic preparation};

\node[inputbox] (radius) at (0,0)
  {\S\ref{sec:reference-radius}\\Comparing radius functions};
\node[inputbox] (form) at (5.35,0)
  {\S\ref{sec:intermediate}\\ Weight of the symplectic form is $2$};
\node[flowbox] (cochains) at (10.7,0)
  {\S\ref{sec:PD-cohomology}-\ref{sec:PD-kunneth}\\K\"unneth in degrees $2,3$};

\node[flowbox] (cutoffs) at (0,-1.6)
  {\S\ref{sec:lower}-\ref{sec:cutoffs}\\Metric and cutoff estimates};
\node[flowbox] (blocks) at (5.35,-1.6)
  {\S\ref{sec:flag}\\Parallel foliations of $C$};
\node[flowbox] (bases) at (10.7,-1.6)
  {\S\ref{sec:PDreview}\\Poisson deformation bases};

\node[flowbox] (bochner) at (0,-3.2)
  {\S\ref{sec:bochner-cutoff}\\Bochner principle};
\node[flowbox] (leaves) at (5.35,-3.2)
  {\S\ref{sec:leaves}\\Algebraic leaves of $C$};
\node[flowbox] (splitting) at (10.7,-3.2)
  {\S\ref{sec:SPD}\\Split product deformation};

\node[flowbox] (cover) at (5.35,-4.8)
  {\S\ref{sec:cover}\\Finite cover $V=\prod_iV_i\to C$};
\node[flowbox] (lift) at (5.35,-6.4)
  {\S\ref{sec:extendcover}\\Deform the finite cover $V$};
\node[flowbox] (casimir) at (5.35,-8.0)
  {\S\ref{sec:tensors}\\Casimir vanishing};
\node[flowbox] (formal) at (5.35,-9.6)
  {\S\ref{sec:recursive}\\Formal rigidity};
\node[flowbox] (conclusion) at (10.7,-9.6)
  {\S\ref{sec:conclusion}: $W\simeq C$\\Descent; algebraization};

\draw[flowarrow] (radius) -- (cutoffs);
\draw[flowarrow] (cutoffs) -- (bochner);
\draw[flowarrow] (bochner.east) -- (2.675,-3.2)
  -- (2.675,-1.6) -- (blocks.west);
\draw[flowarrow] (form) -- (blocks);
\draw[flowarrow] (blocks) -- (leaves);
\draw[flowarrow] (leaves) -- (cover);
\draw[flowarrow] (cover) -- (lift);
\draw[flowarrow] (lift) -- (casimir); 
\draw[flowarrow] (casimir) -- (formal);
\draw[flowarrow] (cochains.east) -- (13.35,0)
  -- (13.35,-3.2) -- (splitting.east);
\draw[flowarrow] (bases) -- (splitting);
\draw[line width=.45pt] (splitting.south) -- (10.7,-8.8)
  -- (5.35,-8.8);
\fill (5.35,-8.8) circle[radius=.8pt];
\draw[flowarrow] (formal) -- (conclusion);
\end{tikzpicture}
\end{center}

\subsection*{Acknowledgements and use of AI}
The origin of this work i.e., these 
canonical torus actions on general symplectic singularities, 
refining the prediction by Kaledin, 
is our joint work with Yoshinori Namikawa 
from May 2024 which culminated in the joint paper 
\cite{NO} in March 2025. 
In {\it op.cit}, we 
essentially proved our main Theorem \ref{thm:rigidity}, 
Corollary \ref{cor:hk} and more, under some 
technical conditions. For this, I am truly grateful to
Y.Namikawa for the discussions and supports.
Indeed, the basic idea of using these Donaldson-Sun $2$-step degenerations for 
hyperK\"ahler metrics and relate to Poisson deformation theory, to prove rigidity
are already in that joint work. These ideas started with some 
fairly simple 
(human) observations on 
deformations of the (fairly classical)  
Kleinian singularities and differential geometric 
considerations on them. 

In the summer of 2026, I began this project as an experiment 
to test and understand 
various abilities and limitations 
of large language models then. 
I asked 
ChatGPT Astra, Claude Fable and their agents services 
if ``they" could remove the technical 
hypotheses in \cite{NO} 
which we unfortunately needed to assume for technical reasons then 
in \cite{NO} 
(mainly smoothability and globalizability assumptions). 

Firstly none of the AI models were able to give immediate answers. 
After substantial ``discussions" with those LLM, which reflects 
some suggestions from my side we obtained through the joint work, 
by mid-September 
2026, 
those LLM started to claim the completion of this technical 
proof. 

From there was a longer step: I then spent 
a month and a half or more, working with this proof to verify and understand: to first 
understand related background materials, read and verify 
the details of all the arguments, 
simplify/expand them, 
which partly relies on those LLM. Then the author 
largely edited the paper reflecting my own preference. 
Thus, most of the process of making this paper substantially 
relies 
on them, though I of course take responsibility. 

At the same time, I also benefitted from various 
discussions with V.Apostolov, Y.Hada, 
M.Hattori, E.Legendre, Y.Namikawa, C.Spotti and S.Sun to who I 
thank warmly. Although this experience has shown 
me the remarkable technical capabilities of 
LLMs, it did not affect my beliefs and 
thoughts on the invaluability of our 
human mathematical activities. 

During this work, the author was partially supported by JSPS 
KAKENHI Grant JP23K25766
(Grant-in-Aid for Scientific Research (B)), JP23KK0249 (Fund for 
the Promotion 
of Joint International Research (Fostering Joint International 
Research))
and JP25H00586 (Grant-in-Aid for Scientific Research (A)). 

\section{Analysis on the Calabi-Yau cone \texorpdfstring{$C$}{C}}\label{sec:analysis}
This section provides the analytic preparation for the core argument in
\S \ref{sec:blocks}; it does not use $W$, $\sigma_W$, or the
degeneration $\cW$. \S \ref{sec:reference-radius} first records the
metric inputs for the original, possibly irrational, Reeb vector $\xi$.
The metric radius is differentiated only on the regular locus; a
reference radius and a comparison inequality provide all compactness
needed near the singular set. The remaining subsections work in the
cone-metric setting of \ref{hyp:metric}, with 
the Reeb and torus conventions in our Introduction section, and
prove Theorem \ref{thm:bochner}. We first state the required cutoff
condition, then establish it using a polynomial metric estimate and a
local energy estimate along the exceptional divisor.

When $C$ has an isolated singularity at the vertex,
the Bochner argument is straightforward: the metric annuli
are compact subsets of the smooth locus, so ordinary
integration by parts applies without any cutoff arguments. 
The main purpose of this section is to extend 
this argument to the non-isolated case. 
Throughout this section, $d^c=\frac{i}{2}(\bar\partial-\partial)$, so
$dd^c=i\partial\bar\partial$. 

\subsection{Radius type functions}
\label{sec:reference-radius}
We use the existence and regularity theorems only through the following
properties. The metric is smooth on $U=C^{\reg}\setminus\{o\}$,
its potential is locally bounded and psh on $C$, and its radius is
uniformly comparable to an explicit reference radius. The comparison
allows us to work with ordinary compact sets without studying the
values of the metric radius on the singular set.

Fix an affine embedding $C\hookrightarrow\C^M$ by eigenfunctions of
the prescribed torus. For a Reeb vector $\zeta\in N_{\mathbb{R}}$, 
let $r_\zeta$ denote 
the ambient reference distance from the origin in
\cite[\S4.3, Notation 3]{NO}, restricted to $C$.
It is continuous, proper, positive off the vertex, and homogeneous
of degree one under the $\zeta$-dilations. We use $r_\xi$ for the
original Reeb vector and $r_\chi$ for an auxiliary quasi-regular element 
$\chi$. 

\begin{lemma}[Relative boundedness of the radius]\label{lem:radius-comparison}
Let $(C,\xi)$ be a K-polystable klt cone of dimension $n\ge2$.
The weak Ricci-flat cone metric constructed in
\cite[Theorem 2.9]{LiSoliton} has a nonnegative locally bounded psh
potential $r^2$ on $C$. By Nghiem's regularity theorem
\cite[Theorem 1.5]{Nghiem}, this potential is smooth on
$U=C^{\reg}\setminus\{o\}$; in particular, $r$ is a smooth positive
radius there, and the associated metric $g$ is smooth and Ricci-flat.
For the reference radius $r_\xi$ above, there is a
constant $\kappa\ge1$ such that
\begin{equation}\label{eq:radius-comparison}
 \kappa^{-1}r_\xi(x)\le r(x)\le\kappa r_\xi(x)\qquad(x\in U).
\end{equation}
\end{lemma}

\begin{remark*}[Continuity of the radius]
As Nghiem writes \cite[\S3.6, p. 34]{Nghiem}, 
one naturally expects the metric radius to be continuous across the
singular locus. 
If $r$ extended continuously to $C$, remaining positive away from the
vertex, then \eqref{eq:radius-comparison} would follow immediately
from the positive minimum and finite maximum of $r/r_\xi$ on the compact
set $\{r_\xi=1\}$ and the common homogeneity.
The argument below obtains this comparison from the known boundedness of 
potential for weighted solitons, without requiring such continuity.
\end{remark*}

\begin{proof}
Choose a quasi-regular Reeb vector $\chi$ in the same torus,
normalized by $A_C(\chi)=n$, and let
\[
 \pi_\chi:C\setminus\{o\}\longrightarrow X_\chi
\]
be the quotient map by the associated $\C^*$-action.
The quotient is a log Fano pair $(X_\chi,D_\chi)$ as is well-known. 
Write $X_\chi^\circ=\pi_\chi(U)$. 

We use the weighted soliton construction whose underlying
geometric correspondence was developed by Apostolov and
Calderbank \cite[Theorem 1]{AC}, with the Sasaki-Einstein
specialization given by Apostolov, Jubert and Lahdili
\cite[Proposition 2 and Lemma 2.6]{AJL};
see also \cite[Remark 2.3]{LiSoliton}. 
By \cite[Theorem 2.9 and its proof, pp. 34-35]{LiSoliton},
the log Fano quotient $(X_\chi,D_\chi)$ admits a bounded
potential $u$ satisfying
\[
 g(\mu_u)\omega_u^{n-1}=c e^{-u}\Omega_0,
 \qquad
 \omega_u=\omega_0+dd^cu,
 \qquad
 g(y)=(n+\langle y,v\rangle)^{-(n+1)},
\]
where $c>0$.
Here $\omega_0$ and $\Omega_0$ are the curvature form
and log volume measure of a torus-invariant reference
Hermitian metric $h_0$ on $A:=-(K_{X_\chi}+D_\chi)$.
The moment map $\mu_u$ is for the compact torus in
$T/\C_\chi^*$, with constants fixed by the natural lift
to $A$, and $v$ is the image of $\xi-\chi$ in its Lie algebra.

Let $\rho_\chi$ be the reference $\chi$-radius induced by $h_0$.
Both $\rho_\chi$ and $r_\chi$ are positive continuous functions away
from the vertex, with the same $\chi$-homogeneity. Their ratio is
therefore bounded above and below on the compact set $\{r_\chi=1\}$,
and hence on $C\setminus\{o\}$.
The level-preserving radius correspondence in
\cite[\S2.2, (79)-(83), and \S2.4, (100)-(105)]{LiSoliton}
gives
\begin{equation}\label{eqn:2.1intermediate}
 \pi_\chi^*u=-2n\log\rho_\chi
 \quad\hbox{on }S:=\{q\in U\mid r(q)=1\},
 \qquad
 \pi_\chi(S)=X_\chi^\circ.
\end{equation}
The boundedness of $u$ in \cite[Corollary 31]{HanLi}
(cf.\ also \cite{EGZ, BBEGZ} and \cite[p. 35]{LiSoliton}) thus bounds
$\rho_\chi$, and hence $r_\chi$, above and away from zero on $S$.
Consequently, $S$ lies in an annulus $\{a\le r_\chi\le b\}$ that is
compact and avoids the vertex. The continuous positive function
$r_\xi$ has positive lower and finite upper bounds there.
Since $r$ and $r_\xi$ have the same $\xi$-homogeneity, their ratio is
invariant under the $\xi$-dilations; the bounds on $S=\{r=1\}\cap U$
therefore give \eqref{eq:radius-comparison} on $U$.

Finally, \cite[Theorem 2.9 and its proof]{LiSoliton} supplies the
locally bounded psh extension of $r^2$ and the Ricci-flat cone equation,
while \cite[Theorem 1.5]{Nghiem} gives smoothness of $r^2$ on $U$.
Since $r>0$ there, $r$ is smooth, and
$\omega=dd^c(r^2/2)$ defines a smooth Ricci-flat K\"ahler metric on $U$.
Compact-torus invariance is preserved by the reference
construction and by these metric constructions.
\end{proof}

In particular, for $0<a<b$, the regular metric annulus satisfies
\[
 \{x\in U\mid a\le r(x)\le b\}
 \subset\{a/\kappa\le r_\xi\le\kappa b\}.
\]
The set on the right is compact and avoids the vertex. We will use
such reference annuli on a resolution, while keeping $r$ in the
differential calculations on $U$.

\subsection{The Bochner type argument and its 
singular extension via cutoff}\label{sec:bochner-cutoff}
We derive the scalar inequality \eqref{eq:bochner}, then show in two integration-by-parts steps that \eqref{eq:good-cutoffs} forces parallelism. All differential operators in this subsection act on the smooth manifold $C^{\reg}\setminus\{o\}$; no boundary condition at the singular set is imposed.

For $0<a<b<\infty$, set
\[
A_{a,b}^{\reg}:=\{x\in C^{\reg}\setminus\{o\}\mid a<r(x)<b\}.
\]
These are open subsets of the regular locus, where $r$ is smooth.
All norms and scalar integrals downstairs use $g$ and $dV_g=\omega^n/n!$. The gradient is defined by $g(\nabla f,V)=df(V)$ for every real tangent vector $V$. In particular $|\nabla f|_g=|df|_g$. The condition $F\in L^2(A_{a,b}^{\reg},dV_g)$ means
\[
\int_{A_{a,b}^{\reg}}|F|^2\,dV_g<\infty,
\]
including the approach to singular limit points of this regular annulus. This is stronger than $F\in L^2_{\mathrm{loc}}(C^{\reg},dV_g)$, which tests only compact subsets avoiding the singular set. Below $F$ will be taken as a fixed function constructed from $\alpha$, not an arbitrary function.

\paragraph{Metric and Laplacian conventions.}
For $\omega=ig_{a\bar b}dz^a\wedge d\bar z^b$, the real metric is $2\operatorname{Re}(g_{a\bar b}dz^a\otimes d\bar z^b)$. We use the nonnegative Hodge Laplacians
\[
\Delta_d:=dd^*+d^*d,\qquad
\Delta_{\bar\partial}:=\bar\partial\bar\partial^*+
\bar\partial^*\bar\partial.
\]
The stars on $d$ and $\bar\partial$ denote their formal adjoints.

Let $\nabla$ be the Levi-Civita connection of $g$. Since $g$ is K\"ahler, its connection on $T^{1,0}$, and hence on holomorphic covariant tensors, is the Chern connection \cite[Chapter VI, \S4]{Demailly}. Note the identity $\nabla R=\id$ for $R=r\partial_r$.

\paragraph{The Bochner type inequality}
Fix a homogeneous $\alpha\in H^0(C,\Omega_C^{[p]})$
as in Theorem \ref{thm:bochner}, and put
\[
u_\alpha:=|\alpha|_g^2,\qquad F:=\log(1+u_\alpha).
\]
Firstly, by the Chern-Bochner formula (cf., e.g., \cite[Chapter V, VIII]{Demailly}),
$\operatorname{Ric}(g)=0$ implies
\[
-\Delta_{\bar\partial}u_\alpha=|\nabla'\alpha|_g^2.
\]

We also estimate the Laplacian of $F$.
Let $\mathscr E=\bigwedge^p(T^{1,0}(C^{\reg}\setminus\{o\}))^*$,
equipped with the Hermitian metric induced by $g$.
For every $Z\in T^{1,0}$, the first-derivative identity and
Cauchy-Schwarz in $\mathscr E$ give
\[
|\partial u_\alpha(Z)|^2
=|\langle\nabla'_Z\alpha,\alpha\rangle_g|^2
\le u_\alpha\,|\nabla'_Z\alpha|_g^2.
\]
Taking the metric trace of this inequality of Hermitian forms on
$T^{1,0}$ yields
\[
|\partial u_\alpha|_g^2\le
u_\alpha|\nabla'\alpha|_g^2.
\]
For a real smooth function $u$ and a real smooth scalar function $\varphi$,
the ordinary chain rule, applied to $d^c\varphi(u)=\varphi'(u)d^cu$, is
\[
dd^c\varphi(u)
 =\varphi'(u)dd^cu+\varphi''(u)\,du\wedge d^cu,
\qquad
\Lambda_\omega(du\wedge d^cu)=|\partial u|_g^2.
\]
Apply this with $\varphi(u)=\log(1+u)$ and use
$\Delta_d=2\Delta_{\bar\partial}$. We obtain
\begin{equation}\label{eq:bochner}
\begin{aligned}
-\Delta_dF
&=2\left(\frac{-\Delta_{\bar\partial}u_\alpha}{1+u_\alpha}
 -\frac{|\partial u_\alpha|_g^2}{(1+u_\alpha)^2}\right)\\
&\ge\frac{2|\nabla'\alpha|_g^2}{1+u_\alpha}
 -\frac{2u_\alpha|\nabla'\alpha|_g^2}{(1+u_\alpha)^2}\\
&=\frac{2|\nabla'\alpha|_g^2}{(1+u_\alpha)^2}\ge0.
\end{aligned}
\end{equation}

The homogeneity assumption gives the following: 
\[
\delta_s^*u_\alpha=u_\alpha,
\qquad RF=0,
\]
where $\delta_s$ denotes dilation by $s>0$. We use these below (to be precise, 
in \eqref{eq:bochner-ibp}). 

\paragraph{On the requirement for the cutoff functions}
If the vertex is the only singular point of $C$,
we may simply take $\chi_L\equiv1$ on each fixed annulus.
Indeed, the radial cutoff $\psi(r)$ already has compact
support in the smooth locus, and the error terms involving
derivatives of $\chi_L$ vanish. For general $C$, however,
the singular locus may meet every such annulus, and we
need additional cutoffs satisfying \eqref{eq:good-cutoffs}. 

Fix $0<a_0<a<b<b_0$. We seek a family of smooth functions 
$\{\chi_L\}_{L\in\mathbb{Z}_{>0}}$ on $A_{a_0,b_0}^{\reg}$, indexed by $L\to\infty$, with $0\le\chi_L\le1$. For each $L$, the function vanishes near the singular set in the closure of $\{x\in C^{\reg}\setminus\{o\}:a\le r(x)\le b\}$, and
$\chi_L(x)\to1$ for every $x\in A_{a,b}^{\reg}$ as $L\to\infty$.
No condition at the radial boundary is imposed on $\chi_L$; a separate fixed function of $r$ will handle it.
For the $F$ just defined, the sufficient estimates are
\begin{equation}\label{eq:good-cutoffs}
\begin{gathered}
F\in L^2(A_{a,b}^{\reg},dV_g),\\
\int_{A_{a,b}^{\reg}}F|\Delta_d\chi_L|\,dV_g
\xrightarrow[L\to\infty]{}0,\qquad
\int_{A_{a,b}^{\reg}}|\nabla\chi_L|_g^2\,dV_g
\xrightarrow[L\to\infty]{}0.
\end{gathered}
\end{equation}
The purpose of the remaining subsections \ref{sec:lower}-\ref{sec:cutoffs} in this section is to prove these necessary estimates. 

\begin{proof}[Proof of Theorem \ref{thm:bochner}, assuming \eqref{eq:good-cutoffs}]
The radial cutoff $\psi(r)$ used below is standard.
The main difficulty in the non-isolated case is to construct
the additional cutoff $\chi_L$ with the estimates
\eqref{eq:good-cutoffs}, which is the purpose of 
Subsections \ref{sec:lower}-\ref{sec:cutoffs}. 
Fix a nonnegative $\psi\in C_c^\infty((a,b))$, independent of $L$, 
and write $\psi$ also for $\psi(r)$. We now give a simple 
estimation using integration by parts via the cut off functions. 

Firstly, since $\psi\chi_L$ has compact support in the smooth locus, the two scalar identities above give
\[
\begin{aligned}
I_L&:=\int_{A_{a,b}^{\reg}}\psi\chi_L(-\Delta_dF)\,dV_g
=-\int_{A_{a,b}^{\reg}}F\Delta_d(\psi\chi_L)\,dV_g\\
&=-\int_{A_{a,b}^{\reg}}\psi F\Delta_d\chi_L\,dV_g
-\int_{A_{a,b}^{\reg}}\chi_LF\Delta_d\psi\,dV_g\\
&\hspace{1em}+2\int_{A_{a,b}^{\reg}}F\langle\nabla\psi,\nabla\chi_L\rangle_g\,dV_g.
\end{aligned}
\]
On the other hand, since
$\langle\nabla F,\nabla\psi\rangle_g
=\frac{\psi'(r)}r\,RF=0,$
a second integration by parts therefore yields
\[
\begin{aligned}
-\int_{A_{a,b}^{\reg}}\chi_LF\Delta_d\psi\,dV_g
&=-\int_{A_{a,b}^{\reg}}\langle\nabla(\chi_LF),\nabla\psi\rangle_g\,dV_g\\
&=-\int_{A_{a,b}^{\reg}}F\langle\nabla\chi_L,\nabla\psi\rangle_g\,dV_g.
\end{aligned}
\]
Thus we have
\begin{equation}\label{eq:bochner-ibp}
\begin{aligned}
I_L=\int_{A_{a,b}^{\reg}}\psi\chi_L(-\Delta_dF)\,dV_g
&=\int_{A_{a,b}^{\reg}}\psi F(-\Delta_d\chi_L)\,dV_g\\
&\quad+\int_{A_{a,b}^{\reg}}F\langle\nabla\psi,\nabla\chi_L\rangle_g\,dV_g.
\end{aligned}
\end{equation}

Now we see that the two errors tend to zero as follows.
The absolute value of the first term on the right of \eqref{eq:bochner-ibp} is at most
\[
\|\psi\|_{L^\infty(a,b)}
\int_{A_{a,b}^{\reg}}F|\Delta_d\chi_L|\,dV_g
\xrightarrow[L\to\infty]{}0
\]
by \eqref{eq:good-cutoffs}. For the second, first use the pointwise bound
$|\langle\nabla\psi,\nabla\chi_L\rangle_g|
\le|\nabla\psi|_g|\nabla\chi_L|_g$, and then Cauchy-Schwarz:
\[
\begin{aligned}
\left|\int_{A_{a,b}^{\reg}}F\langle\nabla\psi,\nabla\chi_L\rangle_g\,dV_g\right|
&\le\|\psi'\|_{L^\infty(a,b)}
\int_{A_{a,b}^{\reg}}F|\nabla\chi_L|_g\,dV_g\\
&\le\|\psi'\|_{L^\infty(a,b)}
\left(\int_{A_{a,b}^{\reg}}F^2\,dV_g\right)^{1/2}
\left(\int_{A_{a,b}^{\reg}}|\nabla\chi_L|_g^2\,dV_g\right)^{1/2}
\xrightarrow[L\to\infty]{}0.
\end{aligned}
\]
Indeed $|\nabla\psi|_g=|\psi'(r)|$ because $|\nabla r|_g=1$.
The third factor tends to zero by the second condition of \eqref{eq:good-cutoffs}.
Hence $I_L\to0$.

By \eqref{eq:bochner}, $\psi\chi_L(-\Delta_dF)\ge0$, and its pointwise limit is $\psi(-\Delta_dF)$. Fatou's lemma gives
\[
0\le\int_{A_{a,b}^{\reg}}\psi(-\Delta_dF)\,dV_g
\le\liminf_{L\to\infty}I_L=0.
\]
Applying \eqref{eq:bochner} once more,
\[
\int_{A_{a,b}^{\reg}}\psi\,
\frac{2|\nabla'\alpha|_g^2}{(1+u_\alpha)^2}\,dV_g=0.
\]
The integrand is continuous and nonnegative on the smooth locus, so $\nabla'\alpha=0$ wherever $\psi>0$. Together with $\nabla''\alpha=0$, this gives $\nabla\alpha=0$ there. Varying the annulus and $\psi$ proves parallelism off the vertex; a smooth vertex is included by continuity. In particular $d\alpha=0$.
\end{proof}

\subsection{The norm function behaviour near the singular locus}\label{sec:lower}
This subsection aims to prove \eqref{eq:F-growth} and \eqref{eq:FL2} on 
the estimates of $F$. 
We obtain \eqref{eq:F-growth}
by combining a known upper bound for the transverse metric with the determinant equation \eqref{eq:coneMA}: an upper bound for the eigenvalues and a lower bound for their product give a lower bound for every eigenvalue.

Choose a projective $T$-equivariant log resolution $\pi:\widetilde C\to C$, isomorphic over $C^{\reg}$, as in \cite[Lemmas 3.1-3.3]{Nghiem}, with reduced SNC exceptional divisor $D=\bigcup_iD_i$. Put $\widetilde C^{\circ}:=\widetilde C\setminus\pi^{-1}(o)$. Let $\kappa\ge1$ be the comparison constant in \eqref{eq:radius-comparison}. For the annuli fixed in \S \ref{sec:bochner-cutoff}, put
\[
\Omega=\pi^{-1}\bigl(\{a_0/\kappa<r_\xi<\kappa b_0\}\bigr),\qquad
K_{a,b}=\pi^{-1}\bigl(\{a/\kappa\le r_\xi\le\kappa b\}\bigr).
\]
The reference radius $r_\xi$ is proper and continuous, so $K_{a,b}$ is compact and contained in the open set $\Omega$. By \eqref{eq:radius-comparison}, $\Omega$ contains $\pi^{-1}(A_{a_0,b_0}^{\reg})$, and $K_{a,b}$ contains the inverse image of $\{x\in C^{\reg}\setminus\{o\}\mid a\le r(x)\le b\}$. Thus estimates on this compact reference annulus apply to the metric annulus used in the Bochner argument. Singular rays may meet this set; no compactness of the whole exceptional divisor is required. Choose smooth Hermitian metrics $H_i$ on $\OO_{\widetilde C}(D_i)$ near $\overline\Omega$, and rescale them so that, for the canonical sections $s_i$,
\begin{equation}\label{eq:divisor-log}
N_D=\prod_i|s_i|_{H_i}\le e^{-1/2},\qquad
\rho=\log N_D^2\le-1,\qquad
dd^c\rho=2\pi[D]-\Theta,
\end{equation}
where $\Theta$ is smooth. 
Thus $\rho$ uses norms of divisor sections, without 
requiring a global equation for $D$. 
We crucially use this $\rho$ for our analysis in 
this section. Note that this $\rho$ is a quasi-psh from the 
above equation \eqref{eq:divisor-log}. 

Below, we prepare a weak $C^2$-estimate of $g$ near the singular 
locus/exceptional divisors. It heavily uses intermediate 
estimates and construction in \cite{Nghiem}, not only the main 
theorem (its Theorem 1.5) in it, as well as \cite{Berman}. 

\begin{lemma}[Polynomial lower bound]\label{lem:lower}
For every compact $K\subset\Omega$ and smooth positive Hermitian metric $H$ on $T\widetilde C$ near $K$, there are $c,M>0$ such that
\[
\pi^*g\ge cN_D^M H\qquad\text{on }K\setminus D.
\]
\end{lemma}
\begin{proof}
For the differential estimates, choose a {\it smooth} $\xi$-
homogeneous 
reference radius $r_{\mathrm{ref}}$\footnote{denoted by $r_\xi$ in {\it op.cit}. } from a positive-weight affine embedding, 
chosen to be $T_c$-invariant as well and the 
level set on $\widetilde C^{\circ}$ is compact 
(see \cite[Lemma 2.2 and \S 2.3]{HS16}). 
The relative potential $u$ defined below is then invariant under both the 
radial and Reeb flows, hence basic.  
Set  
\[
\theta_{\mathrm{ref}}=dd^c\log r_{\mathrm{ref}}^2,\qquad
u=\log(r^2/r_{\mathrm{ref}}^2),\qquad \theta_u=\theta_{\mathrm{ref}}+d_Bd_B^cu.
\]
Here $d_Bd_B^c$ means $dd^c$ in the holomorphic coordinates transverse 
to the foliation generated by $E$. 
By \eqref{eq:radius-comparison}, 
$|\log(r^2/r_\xi^2)|\le2\log\kappa$ on the regular locus.
The positive continuous ratio $r_\xi/r_{\mathrm{ref}}$ is invariant under the
$\xi$-dilations and has positive lower and finite upper bounds on a
compact reference level. Thus $u=\log(r^2/r_{\mathrm{ref}}^2)$ is bounded there. 

On $\widetilde C^{\circ}$, let $\theta'_{\mathrm{ref}}$ be the smooth closed real
$(1,1)$-form in 
\cite[Lemma 3.3]{Nghiem}
\footnote{denoted by $\omega_B$ in {\it loc.cit}}
obtained in 
\cite[Proposition 4.3]{Berman}. 
It is basic for the lifted radial and Reeb fields $R,IR$. 
Its associated Hermitian form is positive definite on 
$T^{1,0}\widetilde C/\C E$, including along $D$ away from $\pi^{-1}(o)$. We will benefit from such a difference when 
$0\in C$ is not an isolated singularity. (Indeed, otherwise, 
we can simply put $\theta'_{\rm ref}:=\theta_{\rm ref}$.) 
The estimates 
\cite[Proposition 3.13 and \S3.5]{Nghiem} then gives, for 
some 
$C,M_0>0$, the pointwise inequality of semipositive Hermitian forms
\begin{equation}\label{eq:transverse-upper}
\theta_u\le CN_D^{-M_0}\theta'_{\mathrm{ref}}\qquad\text{on }K\setminus D.
\end{equation}

To see this, firstly, the uniform estimate of
\cite[Proposition 3.13]{Nghiem} passes to the limit under the local
$C^1_{\rm loc}$-convergence of
\cite[\S3.5, proof of Theorem 1.5]{Nghiem}%
\footnote{denoted by 
$\varphi_{j,\epsilon(j)}\to\varphi_0$ in {\it loc.cit}}
in the sense of currents, giving
\begin{align}\label{eq:transverse-link-bound}
\theta_u
\le C\exp\bigl(-A\psi(x)-n(\Psi_-(x)+u(x))\bigr)
\theta'_{\rm ref}
\quad\text{on }S_{\rm ref}\setminus D,
\end{align}
where $S_{\rm ref}:=\{r_{\rm ref}=1\}
\subset\widetilde C\setminus\pi^{-1}(o)$.
Here, $\theta_u$ satisfies the transverse Monge-Amp\`ere equation
\begin{align}\label{eq:transverse-MA-link}
\theta_u^{n-1}\wedge\eta
=\exp\bigl(n(\Psi_+-\Psi_- -u)\bigr)
(\theta'_{\rm ref})^{n-1}\wedge\eta
\quad\text{on }S_{\rm ref}\setminus D,
\end{align}
with $\eta:=d^c\log r_{\rm ref}^2|_{S_{\rm ref}}$,
and the function $\Psi_-$ of \cite[Lemma 3.5]{Nghiem}
satisfies
\[
n(\Psi_-(x)+u(x))
\ge \sum_{a_i<0}(-a_i)\log|s_i|_{H_i}^2-C
\quad (x\in S_{\rm ref}\setminus D),
\]
where $a_i$ are the discrepancies of $\pi:\widetilde C\to C$,
since $u$ is bounded. 
\footnote{In the case of symplectic singularities, 
the canonical density has no divisorial poles, 
since $C$ is canonical Gorenstein as noted above.
Thus one may take $\Psi_-=0$ in the decomposition
of \cite[Lemma 3.5]{Nghiem}.}
On the other hand, from
\cite[proofs of Lemmas 3.2-3.3]{Nghiem},
\[
\psi(x)=\Phi_B(x)-\log r_{\rm ref}^2(x)
\ge \sum_i b_i\log|s_i|_{H_i}^2-C,
\qquad b_i\ge0,
\quad (x\in S_{\rm ref}\setminus D).
\]
Thus, combining \eqref{eq:transverse-link-bound}
with the above inequalities, we obtain
\[
\exp\bigl(-A\psi-n(\Psi_-+u)\bigr)
\le C\prod_i|s_i|_{H_i}^{-2(Ab_i+a_i^-)}
\le CN_D^{-M_0}
\quad\text{on }S_{\rm ref}\setminus D,
\]
where $a_i^-:=\max\{-a_i,0\}$ and
$M_0\ge 2\max_i\{Ab_i+a_i^-\}$.
Radial transport over bounded flow times then gives
\eqref{eq:transverse-upper} on $K\setminus D$,
for every compact $K\subset\Omega$. 

We next recover the full metric $g$ in local coordinates, 
from the information on $\theta_u$ we got. Inside $K$, we take any point 
and its small enough open neighborhood with 
local coordinates $(z_1,\cdots,z_{n-1},w)$ whose intersection with $D$ is 
$\{z\in P\mid z_1\cdots z_s=0\}.$ 
Thus, outside $D$, we set 
\[
\begin{aligned}
  \ell(z,w)&:=\log(r^2\circ\pi)(z,w)
       =w+\bar w+v(z,\bar z),\\
  \theta_u&=dd_z^c v
       =\sqrt{-1}\sum_{a,b=1}^{n-1}B_{a\bar b}\,dz^a\wedge d\bar z^b,
  \qquad B_{a\bar b}:=\partial_a\partial_{\bar b}v .
\end{aligned}
\]
Here $v$ is a real-valued smooth local function outside $D$. 
All derivatives and norms in the $z$-variables below are 
taken with respect to the Euclidean metric on each slice 
$w={\text{constant}}$. 
With the nonnegative sign convention,
\[
  \Delta_{\rm flat}v
  =-4\sum_{a=1}^{n-1}\partial_a\partial_{\bar a}v
  =-4\sum_{a=1}^{n-1}B_{a\bar a}
  \qquad\text{on }P\setminus D_z, 
\]
where $\Delta_{\rm flat}$ the Laplacian with respect to the standard flat 
metric on $\mathbb{C}^{n-1}$. 
The right-hand side is bounded in absolute value by
$CN_D^{-M_0}$, by \eqref{eq:transverse-upper}.
On finitely many smaller SNC charts covering $K$, choose $\delta=cN_D(p)$
for $p=(z_0,w_0)$ with $c>0$ uniformly small, so that
$B_\delta(z_0)\times\{w_0\}$ avoids $D$ and
$N_D(z,w_0)\asymp N_D(p)$ there.
Since \eqref{eq:radius-comparison} and $\ell=w+\bar w+v$ bound $v$
uniformly, the interior gradient estimate gives
$|dv(z_0)|\le C\bigl(\delta^{-1}\|v\|_\infty+
\delta\|\Delta_{\rm flat}v\|_\infty\bigr)$.
Applying an estimate for
Poisson's equation \cite[\S3.4, (the former inequality of) 
Theorem 3.9]{GT} 
on $B_\delta(z_0)$, we obtain 
\begin{align}\label{eq:flowbox-gradient}
  |dv(z_0)|
  &\le C\left(N_D(p)^{-1}
                   +N_D(p)^{-M_0}\right)
   \le CN_D(p)^{-M_1},
   \end{align}
where $p$ in the local chart outside $D$, $M_1:=\max\{1,M_0\}$. 
We now estimate the full metric $g$. We have 
\[
\begin{aligned}
  \pi^*\omega
  &=\frac12dd^c e^\ell\\
  &=\frac{e^\ell}{2}
       \left(dd^c\ell+\sqrt{-1}\partial\ell\wedge\bar\partial\ell\right)\\
  &=\frac{e^\ell}{2}
       \left\{\theta_u
       +\sqrt{-1}(dw+\partial v)\wedge(d\bar w+\bar\partial v)\right\}.
\end{aligned}
\]
Its Hermitian quadratic form at $p=(z_0,w_0)$ is therefore
\[
  G_p(Z,\zeta)
  =\frac{e^{\ell(p)}}2
       \left\{
       \sum_{a,b=1}^{n-1}B_{a\bar b}(z_0)Z^a\overline{Z^b}
       +|\zeta+
          \sum_{a=1}^{n-1}(\partial_a v)(z_0)Z^a|^2
       \right\},
  \qquad (Z,\zeta)\in\mathbb C^{n-1}\times\mathbb C.
\]
Using $e^\ell=(r\circ\pi)^2\le\kappa^4b_0^2$ 
by \eqref{eq:radius-comparison} and the definition of $\Omega$, 
and 
$|\zeta+\partial v(Z)|^2
 \le2|\zeta|^2+2|\partial v|^2|Z|^2$, we obtain 
\begin{align*}
  G_p(Z,\zeta)
  &\le C\left\{
       N_D(p)^{-M_0}|Z|^2
       +|\zeta|^2+N_D(p)^{-2M_1}|Z|^2
       \right\}\\
  &\le CN_D(p)^{-M_2}(|Z|^2+|\zeta|^2),
  \qquad M_2:=\max\{M_0,2M_1\},
\end{align*}
for $p$ outside $D$. 
Here the first inequality uses \eqref{eq:transverse-upper} and \eqref{eq:flowbox-gradient},
and the second uses $N_D\le1$ from \eqref{eq:divisor-log}.
On the finitely many smaller charts, positivity and smoothness of
$H$ give a uniform comparison
\[
  |Z|^2+|\zeta|^2
  \le C H_p\bigl((Z,\zeta),\overline{(Z,\zeta)}\bigr)
\]
in the small neighborhood of $p$. 
Since these charts cover $K$, by replacing $C$ 
by larger enough constant if necessary, we conclude
\begin{equation}\label{eq:full-upper}
  \pi^*g\le CN_D^{-M_2}H
  \qquad\text{on }K\setminus D
\end{equation}
for some $C>0$. 
Finally, let $\lambda_1\le\cdots\le\lambda_n$ be the Hermitian eigenvalues of $\pi^*g$ relative to $H$. Equations \eqref{eq:coneMA}-\eqref{eq:canmeasure} give
\[
\prod_{j=1}^n\lambda_j=b_H\prod_i|s_i|_{H_i}^{2a_i}
\ge cN_D^{M_3},
\]
with $b_H$ smooth and positive and $M_3\ge0$. By \eqref{eq:full-upper},
\[
\lambda_1=\frac{\prod_j\lambda_j}{\prod_{j=2}^n\lambda_j}
\ge cN_D^{M_3+(n-1)M_2},
\]
which proves the lemma.
\end{proof}

\paragraph{Upper bound and $L^2$-property of $F$.}
The extension result \cite[Theorem 1.4]{GKKP} gives a  
regular $p$-form $\widetilde\alpha=\pi^{[*]}\alpha$ on 
$\widetilde C$. Its coefficients are bounded on $K_{a,b}$, 
so Lemma \ref{lem:lower} implies 
\begin{equation}\label{eq:F-growth}
0\le\widetilde F:=\pi^*\log(1+|\alpha|_g^2)
\le C(1-\rho)\qquad\text{on }K_{a,b}\setminus D.
\end{equation}
The same bound gives the $L^2$ assertion in \eqref{eq:good-cutoffs} directly from the klt condition. In an SNC chart, \eqref{eq:canmeasure} expresses the volume density as a smooth positive factor times $\prod_i|z_i|^{2a_i}$, and $\rho=\sum_i\log|z_i|^2+O(1)$. Combining it with 
$a_i>-1$, 
\begin{align}\label{eq:FL2}
F\in L^2(A_{a,b}^{\reg},dV_g)
\end{align}
follows, which confirms one of \eqref{eq:good-cutoffs}. 

\subsection{Construction of the cutoff functions and necessary estimates}\label{sec:cutoffs}
On $\Omega$, we define 
\[
T_0=dd^c((r^2/2)\circ\pi),\qquad
\mathcal S=T_0^{n-1},\qquad \mu=h_0\wedge\mathcal S,
\]
where $h_0$ is a smooth K\"ahler form near $\overline\Omega$. These are well-defined since $T_0=\pi^*\omega$ outside $D$, so it has 
bounded potential so that we can take Bedford-Taylor  products. A useful observation is that
$[D]\wedge\mathcal S=0$: the exceptional divisor maps to a set of codimension at least two. Consequently smoothing a truncation of $\rho$ creates no limiting mass on $D$. This gives exactly the energy estimate on the region where the cutoff changes.

\begin{lemma}\label{lem:energy}
One has $|\rho|\in L^1_{\mathrm{loc}}(\Omega,\mu)$, $\mu(D)=0$, and
\[
[D]\wedge\mathcal S=0,\qquad
dd^c(\rho\mathcal S)=-\Theta\wedge\mathcal S.
\]
For every compact $K\subset\Omega$,
\begin{equation}\label{eq:energy}
\int_{K\cap\{-2L<\rho<-L\}}
d\rho\wedge d^c\rho\wedge\mathcal S=o(L)
\qquad(L\to\infty).
\end{equation}
We denote its left hand side as $J(L)$, for later convenience. 
\end{lemma}
The above \eqref{eq:energy} will be used to show the remained 
estimates (of \eqref{eq:good-cutoffs}) for cutoff function. 
\begin{proof}
From \cite[Chapter III, Proposition 3.11]{Demailly} 
(which in turn follow from integral by parts) 
it follows that $\int_K |\rho|d\mu<\infty$, so that $\rho \mathcal{S}$ 
is well-defined as a current. 

Locally write $\rho=\log|f_D|^2-\varphi$, where $f_D$
defines $D$ and $\varphi$ is smooth with $dd^c\varphi=\Theta$.
Since $u=(r^2/2)\circ\pi$ is locally bounded and psh,
and the unbounded locus of $\log|f_D|^2$ has real dimension
$2n-2$, \cite[Proposition 3.2(ii)]{AnderssonWulcan}
and the Poincar\'e-Lelong formula give
\begin{align}\label{eq:PLS}
dd^c(\rho\mathcal S)
=2\pi[D]\wedge\mathcal S-\Theta\wedge\mathcal S.
\end{align}
Here $[D]\wedge\mathcal S$ is the Bedford-Taylor product
defined recursively from $[D]$ as follows. 
A key point here is to show its vanishing, by using that $D$ is 
exceptional divisor. For that, 
we concentrate 
around each component $D_i$ of $D$, 
but by summing up, we obtain the desired conclusion. 
For each $D_i$, put $u=(r^2/2)\circ\pi$, using the given psh representative. 
An iterative Bedford-Taylor product gives 
\[
Q_0=[D_i],\qquad Q_{j+1}=dd^c(uQ_j),
\]
with $Q_{n-1}=[D_i]\wedge \mathcal{S}$. 
We show inductively that their pushforwards vanish. 

First, $\pi_*Q_0=0$, since
$\operatorname{codim}_C\pi(D_i)\ge2$. 
Suppose $\pi_*Q_j=0$ for some fixed $j$. 
For a compactly supported positive test form downstairs, choose $M$ bounding  $u$ from above on its compact inverse-image support. 
Since $Q_j$ is positive and $0\le u\le M$ on $\pi^{-1}(\operatorname{supp}\eta)$, we have
\[
0\le\langle\pi_*(uQ_j),\eta\rangle
=\langle uQ_j,\pi^*\eta\rangle
\le M\langle Q_j,\pi^*\eta\rangle
=M\langle\pi_*Q_j,\eta\rangle=0.
\]
Hence $\pi_*(uQ_j)=0$, and commutation of proper pushforward with $dd^c$ gives $\pi_*Q_{j+1}=0$. The last current $Q_{n-1}=[D_i]\wedge\mathcal S$ is a positive measure. A nonnegative test function downstairs equal to one on $\pi(K)$ therefore shows $Q_{n-1}(K)=0$ for every compact $K\subset\Omega$. This proves $[D]\wedge\mathcal S=0$ and the identity \eqref{eq:PLS} 
for $dd^c(\rho\mathcal S)$. 

To prove \eqref{eq:energy}, we first prepare an auxiliary 
smooth convex function $q$ with 
\[
0\le q'\le1,\quad q'=0\text{ on }(-\infty,-3],\quad
q(t)=t\text{ on }[-1/2,\infty),\quad q''\ge c>0\text{ on }[-2,-1].
\]
Set $q_L=Lq(\rho/L)$. This is smooth across $D$, since it is constant near $D$, and we have 
\[
q_L\to\rho, |q_L|\le|\rho|. 
\] 
The chain rule computation with 
$\wedge \mathcal{S}$ gives
\[
dd^c(q_L\mathcal S)
=-q'(\rho/L)\Theta\wedge\mathcal S
+L^{-1}q''(\rho/L)d\rho\wedge d^c\rho\wedge\mathcal S.
\]
Choose $(0\le)\eta\in C_c^\infty(\Omega)$, which is equal to $1$ near $K$. Using the above formula, the integration of the second term of the 
right hand side against $\eta$ gives 
\begin{equation}\label{eq:convex-energy}
\begin{aligned}
L^{-1}\int_\Omega\eta q''(\rho/L)d\rho\wedge d^c\rho\wedge\mathcal S
&=\int_\Omega q_Ldd^c\eta\wedge\mathcal S
+\int_\Omega\eta q'(\rho/L)\Theta\wedge\mathcal S\\
&\longrightarrow\int_\Omega\rho\,dd^c\eta\wedge\mathcal S
+\int_\Omega\eta\Theta\wedge\mathcal S=0.
\end{aligned}
\end{equation}
Both limits follow by dominated convergence, using $|\rho|\in L^1(\mu)$ on the test support and $\mu(D)=0$. 
The last equality follows by integration by parts
for currents, using
$dd^c(\rho\mathcal S)=-\Theta\wedge\mathcal S$
proved above. 
Since the left hand side is nonnegative and $q''\ge c>0$ where $-2L<\rho<-L$, equation \eqref{eq:energy} follows.
\end{proof}

Choose a smooth nondecreasing $h_{\mathrm{cut}}:\R\to[0,1]$, zero on $(-\infty,-2]$ and one on $[-1,\infty)$, and set
\[
\widetilde\chi_L:=h_{\mathrm{cut}}(\rho/L),
\]
which is a priori defined in the domain of $\rho$. However, 
it naturally extends smoothly by $0$ near $D$ and tends to $1$ 
outside $D$. All its derivatives are supported in the ``transition 
region" 
$E_L=\{-2L<\rho<-L\}$. By the chain rule, we have 
\begin{align}\label{eq:ddchiL_chainrule}
dd^c\widetilde\chi_L
=-L^{-1}h_{\mathrm{cut}}'(\rho/L)\Theta
+L^{-2}h_{\mathrm{cut}}''(\rho/L)d\rho\wedge d^c\rho.
\end{align}
For $K=K_{a,b}$ put
$J_L=\int_{K\cap E_L}d\rho\wedge d^c\rho\wedge\mathcal S=o(L)$.
On $E_L$, \eqref{eq:F-growth} gives $\widetilde F\le CL$ for some $C>0$. 
Since $h_{\rm cut}'$ and $h_{\rm cut}''$ are bounded
and supported in $[-2,-1]$, \eqref{eq:ddchiL_chainrule}
implies
\begin{equation}\label{eq:cutoffbound1}
\int_{K\setminus D}\widetilde F\,
\bigl|dd^c\widetilde\chi_L\wedge\mathcal S\bigr|
\le C\mu(K\cap E_L)+CL^{-1}J_L
\longrightarrow0.
\end{equation}
Here the absolute value denotes the total variation
of a signed measure. The first term tends to zero
since $\mu$ is locally finite and $\mu(D)=0$,
while the second tends to zero by $J_L=o(L)$
from Lemma \ref{lem:energy}.

Similarly, the chain rule gives
$d\widetilde\chi_L
=L^{-1}h_{\rm cut}'(\rho/L)d\rho$.
Using the K\"ahler energy identity, we obtain
\begin{equation}\label{eq:cutoffbound2}
\int_{K\setminus D}
|\nabla\widetilde\chi_L|_{\pi^*g}^2\,
\pi^*dV_g
\le CL^{-2}J_L=o(L^{-1}).
\end{equation}

The above preparations are enough to ensure the necessary 
estimates for the cutoff functions, which complete the 
proof of Theorem \ref{thm:bochner}. 

\begin{proof}[Completion of the proof of Theorem \ref{thm:bochner}: 
confirming \eqref{eq:good-cutoffs}]
Recall that we already proved that Theorem 
\ref{thm:bochner} follows once these $\chi_L$ satisfy 
\eqref{eq:good-cutoffs}, in \S \ref{sec:bochner-cutoff}. 

Now, since $\pi$ is an isomorphism over the regular
locus and $\mathcal S=\pi^*\omega^{n-1}$ there,
the identity
\[
(\Delta_d\chi_L)dV_g
=-\frac{2}{(n-1)!}\,
dd^c\chi_L\wedge\omega^{n-1}
\]
shows that \eqref{eq:cutoffbound1} gives the
Laplacian estimate in \eqref{eq:good-cutoffs}.
Likewise, \eqref{eq:cutoffbound2} gives the
gradient estimate there.
The remaining condition $F\in L^2$ was proved
at the end of \S \ref{sec:lower},
using \eqref{eq:F-growth}. 
Summarizing, 
these complete the proof of Theorem \ref{thm:bochner}. 
\end{proof}

This completes the analytic preparation for later 
\S \ref{sec:flag}. We next
prepare the independent Poisson-deformation tool of \S \ref{sec:PD};
both materials are used in the main argument of \S \ref{sec:blocks}.

\section{Poisson cohomology and splitting of deformations}\label{sec:PD}
This section provides an independent algebraic preparation for
the core argument in \S \ref{sec:blocks} from  
Poisson cohomology and its deformation theoretic application. 

\subsection{Some review of Harrison-Poisson cohomology}
\label{sec:PD-cohomology}
Below, in this subsection, 
tensor products without subscripts are over $\C$.
\paragraph{Harrison-Poisson cohomology.}
Poisson cohomology for general Poisson algebras
was developed by Fresse \cite{Fresse} using Harrison cochains \cite{PXHarrison};
see also Ginzburg-Kaledin \cite[Appendix A]{GK}
and Namikawa \cite[\S 2]{NamikawaFlops}. 

We use the full version of the Harrison-Poisson complex
$(\DP^\bullet(T),D)$ of \cite[Appendix A]{GK},
whose differential combines the multiplication and Poisson bracket,
and write
\begin{equation}\label{eq:HP-convention}
 \HP^i(T):=H^i(\DP^\bullet(T),D).
\end{equation}
In low degrees, its cochains are
\[
 \DP^0(T)=T,\qquad
 \DP^1(T)=\Hom_\C(T,T),
\]
\[
 \DP^2(T)=
 \Hom_\C(\Sym^2_\C T,T)
 \oplus\Hom_\C(\bigwedge^2_\C T,T).
\]
Thus a degree-two cochain is a pair $(\mu,\nu)$,
where $\mu$ is symmetric and $\nu$ is alternating,
corresponding to first-order variations of the multiplication
and Poisson bracket, respectively. 
For a unital Poisson algebra, a Harrison-Poisson cochain is
called \emph{normalized} if its value is zero whenever at least
one of its inputs is the unit element $1$.
In degree $1$, this condition simply means $u(1)=0$.
We will occasionally work with such cochains. 
The normalized subcomplex is known to be 
quasi-isomorphic to 
the full complex, by the standard normalization 
of Hochschild cochains \cite{PXHochschild} and its compatibility 
with the Harrison-Poisson differential 
(cf.\ \cite[Chapter X]{MacLaneHomology} and
\cite[Appendix A]{GK}). 
In particular, every class in $\HP^2(T)$ admits
a representative $(\mu,\nu)$ with
$\mu(1,a)=\nu(1,a)=0$ for all $a\in T$.

In lowest degrees, we have 
\begin{equation}\label{eq:HP-low}
 \HP^0(T)=\ZP(T),\qquad
 \HP^1(T)=\Der^{\Pois}(T)/\{X_f\mid f\in T\},
 \qquad X_f=\{f,-\}.
\end{equation}
Here $\ZP(T)$ is the vector space of Casimir elements and
$\Der^{\Pois}(T)$ denotes the vector space of 
derivations which preserves Poisson structures.

\paragraph{First-order Poisson deformations and coordinate changes.}
A marked Poisson deformation over a local Artin ring $\Lambda$
is a flat Poisson $\Lambda$-algebra together with an identification
of its central fiber with $T$.

Consider first $\Lambda=\C[\varepsilon]/(\varepsilon^2)$.
After choosing an identification of the underlying
$\Lambda$-module with $T\otimes_\C\Lambda$,
the deformed multiplication and Poisson bracket take the form
\begin{equation}\label{eq:mpi}
 m_\varepsilon(x,y)=xy+\varepsilon\mu(x,y),
 \qquad
 \pi_\varepsilon(x,y)=\{x,y\}+\varepsilon\nu(x,y),
\end{equation}
where $\mu$ is symmetric and $\nu$ is alternating.
The first-order parts of associativity, the Leibniz rule,
and the Jacobi identity are precisely the cocycle condition
$D(\mu,\nu)=0$.

A change of the chosen module identification,
preserving the marked central fiber, has the form
$\mathrm{id}+\varepsilon u$ for a linear map $u:T\to T$. 
Its effect on the two operations is described by the differential $D$ as 
\begin{equation}\label{eq:gauge-differential}
\begin{aligned}
 (Du)_m(x,y)
   &=x\,u(y)-u(xy)+u(x)\,y,\\
 (Du)_\pi(x,y)
   &=\{x,u(y)\}-\{y,u(x)\}-u(\{x,y\}).
\end{aligned}
\end{equation}
If two deformations differ by 
$(\mu',\nu')-(\mu,\nu)=Du$, then 
$\mathrm{id}-\varepsilon u$ is a Poisson isomorphism
from the first deformation to the second.
Thus first-order marked Poisson deformations
up to isomorphism are classified by $\HP^2(T)$. 
See \cite{GSsurvey}, \cite[Appendix A.8]{GK} 
and \cite[proof of Proposition 5(1)]{NamikawaFlops} 
for the details. 

\paragraph{Small extensions.}
Let $\Lambda'\twoheadrightarrow\Lambda$ have kernel $I$ with
$\m_{\Lambda'}I=0$. For tentative lifts $M,P$ of the two operations,
fix the defect conventions
\begin{equation}\label{eq:defects}
\begin{aligned}
 \mathsf A(x,y,z)&=M(M(x,y),z)-M(x,M(y,z)),\\
 \mathsf L(x,y,z)&=P(M(x,y),z)-M(x,P(y,z))-M(y,P(x,z)),\\
 \mathsf J(x,y,z)&=P(P(x,y),z)+P(P(y,z),x)+P(P(z,x),y).
\end{aligned}
\end{equation}
The deformation formalism of \cite[Lemma A.14]{GK} gives an obstruction
class in $I\otimes\HP^3(T)$ from these defects; it is independent of
the tentative lifts and vanishes precisely when a Poisson lift exists.
Differences of genuine lifts are degree-two cocycles, and are removed
by a marked coordinate change precisely when their $\HP^2$-class is
zero. All changes are linear 
in the $I$-valued corrections because $I^2=0$. 
With \eqref{eq:gauge-differential}, if
$\text{target}-\text{model}=Du$, then $\id-u$ is an isomorphism from
the model to the target, with inverse $\id+u$.

\subsection{The Ginzburg-Kaledin K\"unneth formula}
\label{sec:PD-kunneth}
For Poisson algebras $A,B$ put $R=A\otimes B$, with product bracket
\begin{equation}\label{eq:product-bracket}
 \{a\otimes b,a'\otimes b'\}
 =\{a,a'\}_A\otimes bb'+aa'\otimes\{b,b'\}_B.
\end{equation}
We use the following K\"unneth type formulae in \cite[Lemma A.5 and
Proposition A.11(i)]{GK}: the natural comparison
of full Harrison-Poisson complexes gives a quasi-isomorphism 
\begin{equation}\label{eq:product-complexes}
 \DP^\bullet(A)\otimes\DP^\bullet(B)\simeq\DP^\bullet(R),
\end{equation}
and the induced natural isomorphism of commutative algebras: 
\begin{equation}\label{eq:kunneth}
 \bigoplus_{p+q=n}\HP^p(A)\otimes\HP^q(B)
 \xrightarrow{\ \sim\ }\HP^n(R).
\end{equation}

For a special case of $n=2$, further 
if $\pi_B=0$, $\HP^0(A)=\C$ and $\HP^1(A)=0$, then 
\eqref{eq:kunneth} gives
\begin{equation}\label{eq:HP2}
 \HP^2(A\otimes B,\pi_A+0)
 \simeq\HP^2(A)\otimes B\oplus\HP^2(B,0).
\end{equation}
Under the same hypotheses, $\HP^0(B,0)=B$ and
$\HP^1(B,0)=\Der_\C(B)$: all Hamiltonian derivations of the zero
bracket vanish. The degree-three case of \eqref{eq:kunneth} gives
\begin{equation}\label{eq:HP3}
 \begin{aligned}
 \HP^3(A\otimes B,\pi_A+0)
 &\simeq \HP^3(A)\otimes B\\
 &\qquad\oplus(\HP^2(A)\otimes\Der_\C(B))
             \oplus\HP^3(B,0).
 \end{aligned}
\end{equation}
We use \eqref{eq:HP2} to compare genuine lifts in
Theorem \ref{thm:SPD}, and the middle summand of \eqref{eq:HP3} 
in Step 2 of the proof of Theorem \ref{thm:SPD} 
to confirm Casimir property of some lift. 

As its preparation, 
here we record explicitly the mixed summand of the
K\"unneth isomorphism \eqref{eq:HP3}.
Below, 
${\rm pr}_{2,1}$ denotes 
the projection from the left hand side of \eqref{eq:HP3} to 
the middle summand of 
\eqref{eq:HP3} and 
\[
\Phi_{2,1}:
\HP^2(A)\otimes\Der_\C(B)
\longrightarrow \HP^3(A\otimes B)
\]
denote the restriction of the inverse of
\eqref{eq:HP3} to its middle summand. 
Then, 
for a degree-two cocycle
$\eta=(\mu,\nu)$ on $A$ which maps $1$ to $0$, and a derivation
$c\in\Der_\C(B)$, 
$$\Phi_{2,1}([\eta]\otimes c)=[\Omega_{\eta,c}]$$ 
holds, with the degree-three cocycle 
\[
\Omega_{\eta,c}
=(\mathsf A_{\eta,c}=0,
  \mathsf L_{\eta,c},
  \mathsf J_{\eta,c}),
\]
where for $x_i=a_i\otimes b_i$, 
\begin{equation}\label{eq:mixed-product-representative}
 \begin{aligned}
 \mathsf L_{\eta,c}(x_1,x_2,x_3)
  &=\mu(a_1,a_2)a_3\otimes b_1b_2c(b_3),\\
\mathsf J_{\eta,c}(x_1,x_2,x_3)
&=\nu(a_1,a_2)a_3\otimes b_1b_2c(b_3)\\
&+\nu(a_2,a_3)a_1\otimes b_2b_3c(b_1)\\
&+\nu(a_3,a_1)a_2\otimes b_3b_1c(b_2).
\end{aligned}
\end{equation}
These are the low-degree expressions of the product comparison in
\cite[Appendix A.3, (A.10)-(A.12)]{GK}, in the multiplication-and-bracket
convention fixed above. In the case 
$x=a\otimes1$, $y=a'\otimes1$, $z=1\otimes b$, its components are 
\[
 \mathsf L_{\eta,c}(x,y,z)=\mu(a,a')\otimes c(b),\qquad
 \mathsf J_{\eta,c}(x,y,z)=\nu(a,a')\otimes c(b).
\]
Thus a sum of these representatives has mixed component
$\sum_j[\eta_j]\otimes c_j$ under $\pr_{2,1}$.

\subsection{Poisson deformations of conical symplectic 
varieties}\label{sec:PDreview}

The following is a consequence of
\cite[Corollary 2.5, Theorem 5.1 and \S5.4]{NamikawaAffine},
\cite[Theorem 2.1 and Corollary 2.4]{NamikawaNotes},
and the equivariant theory of \cite{Rim}.

\begin{theorem}[cf., \cite{Rim, NamikawaAffine, NamikawaNotes}]
\label{prop:namikawa}
Let $Y$ be a connected conical symplectic variety,
equipped with a torus action and a positive Reeb vector $\xi$
for which the symplectic form is homogeneous of positive weight.
Put $A=\Gamma(Y,\OO_Y)$.
The formal universal Poisson deformation has a smooth
finite-dimensional base
\[
 D_A=\C[[h_1,\ldots,h_s]],\qquad s=\dim\HP^2(A),
\]
whose coordinates can be chosen homogeneous of positive
$\xi$-weights.
Compatible torus and commuting finite symplectic group actions
on $Y$ lift to the family.
\end{theorem}

We shall also later use the following. 
\begin{lemma}\label{lem:HP(A)}
For the coordinate ring $A$ of a conical symplectic variety as 
above, 
\begin{equation}
 \HP^0(A)=\C,\qquad \HP^1(A)=0.
\end{equation}
\end{lemma}
\begin{proof}
$\HP^0(A)=\C$ follows from the connected dense 
symplectic 
locus $U:=Y^{\reg}$. For the vanishing of 
$\HP^1(A)$, let $v$ be a Poisson
derivation and put $\beta=\iota_v\sigma$, where $\sigma$ is the
symplectic form on $U$. Since $\mathcal L_v\sigma=0$ and
$d\sigma=0$, Cartan's formula gives $d\beta=0$.
Choose a good integral $\C^*$-action in the given torus,
with Euler field $E_0$, and decompose the algebraic form
into its finitely many weight components 
$\beta=\sum_w\beta_w$ with $d\beta_w=0.$ 
For $w\ne0$, Cartan's formula yields 
$\beta_w=d\!\left(\frac{\iota_{E_0}\beta_w}{w}\right).$ 
The primitive is regular on $U$ and extends to an element of $A$
by normality. Since $Y$ is klt, \cite[Theorem 1]{Braun} therefore 
implies 
that $\pi_1(U)$ is finite.
Thus the period homomorphism 
$[\gamma]\mapsto\int_\gamma\beta_0$ is zero,
and integration gives a single-valued holomorphic function $h$
on $U$ with $dh=\beta_0$. 
Normality extends $h$ to $Y$ as a holomorphic (analytic) function. 
Since $d(E_0h)=\mathcal L_{E_0}(dh)=\mathcal L_{E_0}\beta_0=0,$ 
hence $E_0h$ is constant; evaluating at the fixed point $o$
shows that this constant is zero. 
Consequently $h$ is invariant under the contracting action,
and continuity at $o$ gives $h(x)=h(o)$ for every $x$.
Thus $\beta_0=0$. 
Together with the regular primitives of the nonzero-weight
components, this shows that $\beta=df$ for some $f\in A$.
Every Poisson derivation is therefore Hamiltonian,
proving $\HP^1(A)=0$. 
\end{proof}

\subsection{On Poisson deformations of (tensor) products}\label{sec:SPD}
Write $\cU_A$ for the flat formal Poisson $D_A$-algebra with central fiber $A$. The $h_j$ are coordinates on the base, not elements of $A$. 
Note that $\mathcal{U}_A\simeq A[[h_1,\cdots,h_s]]$ 
as $D_A$-module 
by Nakayama's lemma. 
If $m_h,\Theta_h$ are its two operations under such a fixed $D_A$-module 
trivialization of $\mathcal U_A$, 
the multiplication and Poisson bracket can be expanded as
\begin{equation}\label{eq:universal-expansion}
\begin{aligned}
m_h(a,a')
 &=aa'+\sum_{j=1}^s h_j\mu_j(a,a')
   +\sum_{|\alpha|\ge2}h^\alpha m_\alpha(a,a'),\\
\Theta_h(a,a')
 &=\pi_A(a,a')+\sum_{j=1}^s h_j\nu_j(a,a')
   +\sum_{|\alpha|\ge2}h^\alpha\Theta_\alpha(a,a').
\end{aligned}
\end{equation}
Here $\alpha\in\mathbb Z_{\ge0}^s$ is a multi-index,
$h^\alpha=\prod_jh_j^{\alpha_j}$, and the first-order
coefficients $(\mu_j,\nu_j)$ represent the Kodaira-Spencer
classes. Thus, their derivatives represent the Kodaira-Spencer type 
basis 
\begin{equation}\label{eq:KS}
\kappa_j=\bigl[(\partial_{h_j}m_h|_0,\partial_{h_j}\Theta_h|_0)\bigr]\in\HP^2(A).
\end{equation}
The tangent vectors $\partial/\partial h_j$, rather than the functions $h_j$, map to these classes.

Below, we consider a deformation of the tensor product of 
Poisson algebra in the form 
$(A\otimes B,\pi_A+0_B)$, 
where 
$A,B$ are finitely generated Poisson algebras, with zero bracket on 
$B$. As a priori special ``product" type deformation, we can 
consider the following deformation: 
for a Poisson deformation $\cB$ of $(B,0_B)$ over $\C[[t]]$,  
we consider a homomorphism 
\begin{equation}\label{eq:parameter-map}
\varphi:D_A\longrightarrow\cB,\qquad h_j\longmapsto f_j, 
\end{equation}
and then associate the deformation 
$$\cU_A\cten_{D_A,\varphi}\cB,$$ 
where the completed tensor product means $t$-adic completion of the 
usual tensor product as algebras. 
The following shows that, 
under mild assumptions, this is the form any Poisson 
deformation of the tensor product $(A\otimes B,\pi_A+0_B)$ should 
take. 

\begin{theorem}[Splitting of Poisson deformations]\label{thm:SPD}
Under the above setup, 
assume $\HP^0(A)=\C$, $\HP^1(A)=0$, and there is 
a smooth finite-dimensional 
formal miniversal family $\cU_A$ over $D_A$. Every marked $t$-adically complete and separated flat formal Poisson deformation $\cD$ over $\C[[t]]$ of $(A\otimes B,\pi_A+0_B)$ is isomorphic to
\begin{equation}\label{eq:SPD}
\cD\simeq\cU_A\cten_{D_A,\varphi}\cB,\qquad
\varphi(h_j)=f_j\in t\cB\cap\ZP(\cB),
\end{equation}
where $\cB$ is a flat formal Poisson deformation of $(B,0)$. The isomorphism is the identity on the marked central fiber. 

If compatible torus and commuting finite-group actions preserve 
the Poisson subalgebras $A, B$ of $(A\otimes B,\pi_A+0_B)$ 
and the given deformation $\mathcal{D}$, the presentation can be  chosen equivariantly. A torus may scale the bracket and move $t$; the finite group preserves the bracket. The map \eqref{eq:parameter-map} and the induced model action then respect these same actions.
\end{theorem}

\begin{remark}\label{remark:productcounterexample}
The conclusion does {\it not} hold for arbitrary Poisson algebras. 
Take $A=\C[x]$ and $B=\C[y]$, both with zero Poisson brackets. 
Then 
\[
\mathcal D=\C[x,y][[t]],\qquad \{x,y\}=t
\]
is a flat formal Poisson deformation of $(A\otimes B,0)$. On the other hand, its 
each factor (of Krull dimension $1$) 
has only trivial formal Poisson deformations: 
its underlying commutative algebra is formally rigid, 
and every bracket in one variable is zero since $\{x,x\}=0, \{y,y\}=0$. 
Hence the right hand side of \eqref{eq:SPD} would always have 
zero Poisson bracket, contradicting $\{x,y\}=t$. 
Here $\HP^0(A)=A$ and $\HP^1(A)=A\,\partial_x$, 
so the hypotheses on $A$ exclude precisely this type 
of mixed deformation. 
\end{remark}

\begin{proof}
Put
\[
 \Lambda_n:=\mathbb C[t]/(t^{n+1}),\qquad
 \mathcal D_n:=\mathcal D/t^{n+1}\mathcal D.
\]
We construct the presentation 
\eqref{eq:SPD} 
and its marked isomorphism
simultaneously over all $\Lambda_n$. At order $n$, the data are a flat
Poisson deformation $\mathcal B_n$ of $(B,0)$, elements
\[
 f_{j,n}\in t\mathcal B_n\cap Z_{\mathrm{Pois}}(\mathcal B_n)
 \quad(1\leq j\leq s),\qquad \varphi_n(h_j)=f_{j,n},
\]
and a marked Poisson isomorphism
\begin{equation}\label{eq:psin}
 \psi_n:\mathcal E_n:=
 \mathcal U_A\widehat\otimes_{D_A,\varphi_n}\mathcal B_n
 \xrightarrow{\ \sim\ }\mathcal D_n.
\end{equation}
These data must reduce to the previously chosen data at every lower order.
For $n=0$ take $\mathcal B_0=B$, $f_{j,0}=0$, and $\psi_0=\mathrm{id}$.
At a finite order the completed base change means substitution in the 
parameters, followed by reduction modulo $t^{n+1}$.

We use the degree-two K\"unneth formula \eqref{eq:HP2},
\[
 \mathrm{HP}^2(A\otimes B,\pi_A+0)
 \simeq \mathrm{HP}^2(A)\otimes B\ \oplus\ \mathrm{HP}^2(B,0),
\]
and the small-extension calculus recalled in
\S \ref{sec:PD-cohomology}. For the mixed obstruction in Step 2,
we also use the degree-three formula \eqref{eq:HP3} and the
representatives \eqref{eq:mixed-product-representative}.

\medskip
\noindent\textbf{Step 1: take tentative lifts of all the data.}
\label{splitrev:step1}
Assume the construction has been made over $\Lambda_n$, and set
\[
 \Lambda':=\Lambda_{n+1},\qquad
 I:=\ker(\Lambda'\longrightarrow\Lambda_n)=(t^{n+1}).
\]
In particular, $tI=I^2=0$. Choose module identifications compatible with
the already fixed marking and with $\psi_n$, and lift the underlying
module of $\mathcal B_n$ to
\[
 B':=B\otimes\Lambda',\qquad R':=A\otimes B',\qquad R_0:=A\otimes B.
\]
Choose arbitrary unit-preserving lifts on $B'$ of the symmetric
multiplication and alternating bracket of $\mathcal B_n$, and choose
$f'_j\in tB'$ lifting $f_{j,n}$. These are only tentative operations:
$B'$ has not yet been made a Poisson algebra. Write
$f'=(f'_1,\ldots,f'_s)$.

Choose, once and for all, a unit-preserving $D_A$-module trivialization
of $\mathcal U_A$, and denote its multiplication and bracket by
$m_h,\Theta_h$. Substituting $h_j=f'_j$ gives $m_{f'},\Theta_{f'}$;
only finitely many parameter terms contribute since $t^{n+2}=0$.
Define tentative operations on $R'$ by
\begin{equation}\label{eq:temporaryMP}
 \begin{aligned}
 M_{\rm pre}(a\otimes b,a'\otimes b')&=m_{f'}(a,a')\,bb',\\
 P_{\rm pre}(a\otimes b,a'\otimes b')&=\Theta_{f'}(a,a')\,bb'
                  +m_{f'}(a,a')\,\{b,b'\}_{B'}.
 \end{aligned}
\end{equation}
Here multiplication of the $B'$-coefficients uses the tentative
multiplication on $B'$. Fix a parenthesization when evaluating parameter
monomials. Reassociations involving a positive-order parameter have zero
defect: the associator takes values in $IB'$ and is $\Lambda'$-trilinear,
so a parameter in $tB'$ kills it by $tI=0$.

\label{splitrev:obstruction}
These operations 
modulo $I=(t^{n+1})$ 
are exactly the prescribed model modulo $(t^n)$. 
Their associativity, Leibniz and Jacobi defects therefore take values in
\[
 IR'=I(A\otimes B')\simeq I\otimes R_0.
\]
They give the degree-three obstruction cocycle
$\mathrm{ob}\in I\otimes\mathrm{DP}^3(R_0)$ of \S \ref{sec:PD-cohomology}. Since
$\mathcal D_{n+1}$ already exists and lifts $\mathcal D_n$, the marked
isomorphism $\psi_n$ identifies it with a genuine lift of our order-$n$
model. Hence
\[
 [\mathrm{ob}]=0\quad\text{in } I\otimes\mathrm{HP}^3(R_0).
\]
This ensures that some correction on the whole module 
exists. However, it does
{\it not} yet ensure that a correction preserves the partial-
product form like \eqref{eq:temporaryMP}. 
Also note that, as we use throughout the proof, a $\Lambda'$-
multilinear map with values in $IR'$ 
(or $IB'$) depends only on its inputs modulo $t$, because 
$tI=0$.

\medskip
\noindent\textbf{Step 2: make $B'$ a Poisson algebra and $f'_j$ Casimir.}

Choose a linear functional $\ell:A\to\mathbb C$ with $\ell(1)=1$;
it need not be multiplicative or Poisson. Restricting all inputs to
$1\otimes B$ and applying $\ell\otimes\mathrm{id}_B$ to the output
defines a chain map
\[
 \rho_\ell:\mathrm{DP}^{\bullet}(R_0)
       \longrightarrow\mathrm{DP}^{\bullet}(B,0).
\]
Indeed, multiplication by a pure $B$ input acts only on the second
factor, and its Poisson action on all of $R_0$ is zero. Thus both parts
of the cochain differential commute with this restriction and
projection. Extend $\rho_\ell$ by $\mathrm{id}_I$ on coefficients.

Consequently $\rho_\ell(\mathrm{ob})$ is precisely
the defect of the tentative operations on $B'$. Its class is zero,
since $[\mathrm{ob}]=0$. We can therefore choose normalized corrections
\[
 \epsilon_m,\epsilon_\pi:B\times B\longrightarrow IB'
                  \simeq I\otimes B
\]
which make $B'$ a genuine Poisson lift of $\mathcal B_n$;
they vanish whenever either input is $1$.
Extend them $\Lambda'$-bilinearly by reducing inputs modulo $t$,
and add them to the tentative operations on $B'$.
Using the corrected operations on $B'$ and the same 
lifts $f'_j$, define operations $(M,P)$ on $R'$ by the
formulas (3.18).
Let $\mathrm{ob}$ now denote the defect cocycle of this new pair.
Since the pair still lifts the prescribed order-$n$ model,
its obstruction class remains zero.
The multiplication $M$ is associative, being the base change
of $m_h$.
It remains to verify the Leibniz and Jacobi identities for $P$. 

\smallskip
For $b\in B$, choose a lift $\widetilde b\in B'$ and put
\[
 c_j(b):=\{f'_j,\widetilde b\}_{B'}\in IB'\simeq I\otimes B.
\]
The value lies in $IB'$ since $f_{j,n}$ is Casimir. 
It is independent 
of the lift of $b$: replacing $\widetilde b$ by $\widetilde b+tb_1$
changes the value by $t\{f'_j,b_1\}_{B'}\in tIB'=0$.
Also, changing $f'_j$ by an element of $IB'$ has no effect, because
$\{B',B'\}_{B'}\subset tB'$. Changing the lifted bracket by an
$IB'$-valued correction has no effect either, since $f'_j\in tB'$.
The Leibniz rule in $B'$ then gives a $\mathbb{C}$-derivation on $B$ as
\[
 c_j\in\mathrm{Der}_{\mathbb C}(B,I\otimes B)
        =I\otimes\mathrm{Der}_{\mathbb C}(B).
\]
We want to show that these $c_j$ vanish, so that the chosen $f'_j$
are already Casimir elements.
Put
\[
 (\mu_j,\nu_j):=
 \left(\left.\partial_{h_j}m_h\right|_0,
       \left.\partial_{h_j}\Theta_h\right|_0\right),
 \qquad \kappa_j=[(\mu_j,\nu_j)]\in\mathrm{HP}^2(A),
\]
so that the $\kappa_j$ form the basis of $\HP^2(A)$ in \eqref{eq:KS}.
For $x=a\otimes1$, $y=a'\otimes1$ and
$z=1\otimes\widetilde b$, the 
mixed term of the defect of Leibniz rule and the 
defect of the Jacobi rule of 
$(M,P)$ of $(A\otimes B)\otimes \Lambda'$ 
\eqref{eq:temporaryMP} 
are 
\[
\begin{aligned}
L(x,y,z)&=\sum_j\mu_j(a,a')c_j(b),\\
J(x,y,z)&=\sum_j\nu_j(a,a')c_j(b).
\end{aligned}
\]
Indeed, the higher order terms vanish since
$f'_j\in tB'$ and $tI=0$.

Since $B'$ is Poisson, the remaining defects come from bracketing
the coefficients of $m_{f'},\Theta_{f'}$ with $B'$.
On arbitrary pure tensors, the same expansion gives exactly the
sum of the representatives \eqref{eq:mixed-product-representative}
for $\eta_j=(\mu_j,\nu_j)$ and $c_j$; terms containing both $c_j$
and a $B'$-bracket vanish since $\{B',B'\}_{B'}\subset tB'$.
Consequently the mixed component in \eqref{eq:HP3} is
\begin{equation}\label{eq:SPD-mixed-obstruction}
 [\mathrm{ob}]_{\mathrm{mix}}
 :=(\mathrm{id}_I\otimes\pr_{2,1})[\mathrm{ob}]
 =\sum_j\kappa_j\otimes c_j
 \in I\otimes\HP^2(A)\otimes\Der_{\mathbb C}(B).
\end{equation}
Since $\mathcal D_{n+1}$ is a genuine lift, $[\mathrm{ob}]=0$,
so its mixed component is zero. The $\kappa_j$ are linearly
independent, and hence $c_j=0$ for every $j$.

Thus every $f'_j$ is a Casimir element in $B'$.
The map
$D_A\to B'$, $h_j\mapsto f'_j$, consequently takes values in the
Poisson center, and
$\mathcal U_A\widehat\otimes_{D_A,h_j\mapsto f'_j}B'$
has a genuine Poisson algebra structure over $\Lambda'$. 

\medskip
\noindent\textbf{Step 3: obtain an isomorphism with $\mathcal D_{n+1}$ by modifying the lifts.}
\label{splitrev:step3}
At the end of Step 2, $(R',M,P)$ is a genuine Poisson lift.
We must still arrange an isomorphism with $\mathcal D_{n+1}$ which
extends the specified $\psi_n$.

\smallskip
\noindent\emph{Fix the comparison and the sign of the difference.}
\label{splitrev:comparison}
Choose a unit-preserving $\Lambda'$-module isomorphism
$q:R'\xrightarrow{\sim}\mathcal D_{n+1}$ which, modulo 
$t^{n+1}$ is $\psi_n$ \eqref{eq:psin}. 
Such a 
lift is obtained 
by lifting a module basis and fixing the unit. It is not yet 
required 
to preserve multiplication or bracket. Pull back the target operations:
\[
 m_{\mathrm{tar}}(x,y):=q^{-1}(q(x)q(y)),\qquad
 \pi_{\mathrm{tar}}(x,y):=q^{-1}(\{q(x),q(y)\}_{\mathcal D_{n+1}}).
\]
The subindex ${\rm tar}$ stands for ``target", meaning the 
$\mathcal{D}$-side. 
Keep this comparison fixed during the remaining corrections. Put
\[
 \Delta:=(m_{\mathrm{tar}}-M,\pi_{\mathrm{tar}}-P)
          \in I\otimes\mathrm{DP}^2(R_0).
\]
The value of each component lies in $IR'\simeq I\otimes R_0$, and it depends only on 
inputs modulo $t$. Subtracting the identities for the two genuine
Poisson structures gives $D\Delta=0$ 
(see the explanation after \eqref{eq:mpi}). 
Here $D$ is the Harrison-Poisson differential of 
$R_0$ as defined in 
\S \ref{sec:PD-cohomology}.

\smallskip
\noindent\emph{Modify $f'_{j}$ and the Poisson algebra structure on $B'$}
\label{splitrev:class}
Use \eqref{eq:HP2} to write
\[
 [\Delta]=\sum_{j=1}^s\kappa_j\otimes g_j+\delta_B,
 \qquad g_j\in I\otimes B,\quad
 \delta_B\in I\otimes\mathrm{HP}^2(B,0).
\]
The coefficients $g_j$ correct the parameters, while $\delta_B$
corrects the Poisson algebra structure on $B'$. 
These are corrections between 
genuine lifts, unlike the corrections of Step 2 which removes the obstruction. 
Choose a normalized cocycle
\[
 \eta_B=(\alpha_B,\beta_B)\in I\otimes\mathrm{DP}^2(B,0),
 \qquad D\eta_B=0,\qquad [\eta_B]=\delta_B.
\]
Thus $\alpha_B$ is symmetric, $\beta_B$ is alternating, and each
vanishes when an input is $1$. Regard their values, and the $g_j$, as
elements of $IB'$, and extend the cochains to $B'$ by reducing inputs
modulo $t$. Writing $m'_B,\pi'_B$ for the current operations on $B'$,
set
\[
 m_B^+:=m'_B+\alpha_B,\qquad
 \pi_B^+:=\pi'_B+\beta_B,\qquad f_j^+:=f'_j+g_j.
\]
\label{splitrev:corrections}
The new operations satisfy the Poisson identities because 
$D\eta_B=0$. The unit and 
the reduction to $\mathcal B_n$ are unchanged. 

These corrections preserve centrality. Indeed,
$\pi'_B(B',B')\subset tB'$, $g_j\in IB'$ and $f_j^+\in tB'$, so
for every $b'\in B'$,
\[
 \pi'_B(g_j,b')=0,\qquad \beta_B(f_j^+,b')=0.
\]
The first equality uses $I\,tB'=0$ and bilinearity;
the second uses $f_j^+\equiv0\bmod t$, since $\beta_B$ is
$IB'$-valued and $\Lambda'$-bilinear,
while $f_j^+\in tB'$ and $tI=0$. Hence 
\[
 \pi_B^+(f_j^+,b')
  =\pi'_B(f'_j,b')+\pi'_B(g_j,b')+\beta_B(f_j^+,b')=0,
\]
i.e., $f_j^+$ is still a Casimir element. 
Consequently
\[
 \mathcal B_{n+1}:=(B',m_B^+,\pi_B^+),\qquad
 f_{j,n+1}:=f_j^+,\qquad \varphi_{n+1}(h_j)=f_j^+
\]
define corrected partial-product data extending those at 
order $n$. Note that this correction reflects the 
Poisson algebra structure on $\mathcal{D}_{n+1}$ through 
$\Delta$, so that new $\Delta$ would be {\it cohomologically} zero. 
Below, morally speaking, we aim to make this $\Delta$ vanish 
(not only as a class) after some ``coordinate change".

\smallskip
\noindent\emph{Compare the corrected model with the target $\mathcal{D}_{n+1}$.}
\label{splitrev:representatives}
Let $(M^+,P^+)$ be the operations on $A\otimes B'$ 
induced by these corrected data, 
in the same module trivialization of $\mathcal U_A$. That is, 
for $a\in A, b\in B'$, 
\[
\begin{aligned}
M^+(a\otimes b,a'\otimes b')
 &:=m_{f^+}(a,a')\,m_B^+(b,b'),\\
P^+(a\otimes b,a'\otimes b')
 &:=\Theta_{f^+}(a,a')\,m_B^+(b,b')\\
 &\quad+m_{f^+}(a,a')\,\pi_B^+(b,b'),
\end{aligned}
\]
where $f^+:=(f_1^+,\ldots,f_s^+)$. 
As the effect of the modification on the structure of $B'$ (to $\mathcal{B}_{n+1}$) 
we made earlier, we obtain 
\[
 \begin{aligned}
 m_{f'+g}(a,a')-m_{f'}(a,a')&=\sum_j\mu_j(a,a')g_j,\\
 \Theta_{f'+g}(a,a')-\Theta_{f'}(a,a')&=\sum_j\nu_j(a,a')g_j,
 \end{aligned}
\]
as $f'$ lies in $tB'$, $g_j\in IB'$ and $tI=0$. For 
$a,a'\in A$ and $b,b'\in B$ the expansion gives exactly 
\begin{align}
(M^+-M)(a\otimes b,a'\otimes b')
 &=\sum_j\mu_j(a,a')\otimes(g_jbb')
   +aa'\otimes\alpha_B(b,b'),
   \label{eq:splitrev-product-m}\\
(P^+-P)(a\otimes b,a'\otimes b')
 &=\sum_j\nu_j(a,a')\otimes(g_jbb') \notag\\
 &\quad+\pi_A(a,a')\otimes\alpha_B(b,b') \notag\\
 &\quad+aa'\otimes\beta_B(b,b').
   \label{eq:splitrev-product-p}
\end{align}
Here the second and third terms on the right-hand side of
\eqref{eq:splitrev-product-p} arise from
$\Theta_{f'}\alpha_B$ and $m_{f'}\beta_B$,
respectively. 
Denote the pair of \eqref{eq:splitrev-product-m} and 
\eqref{eq:splitrev-product-p} by $L(g,\eta_B)$. 

The first terms on the right-hand sides of
\eqref{eq:splitrev-product-m} and
\eqref{eq:splitrev-product-p} together represent,
under the K\"unneth isomorphism \eqref{eq:HP2},
the class
\[
\sum_j\kappa_j\otimes g_j
\in I\otimes_\C\HP^2(A)\otimes_\C B.
\]
The remaining terms on the right-hand sides of
\eqref{eq:splitrev-product-m} and
\eqref{eq:splitrev-product-p} together are precisely
the infinitesimal changes in the multiplication and
bracket of $A\otimes B$ induced by deforming only
$B$ through $\eta_B=(\alpha_B,\beta_B)$.
Hence they represent the class
\[
\delta_B=[\eta_B]\in I\otimes_\C\HP^2(B,0)
\]
under the K\"unneth isomorphism \eqref{eq:HP2}. 
Combining these together, 
it follows that
\begin{align}\label{eq:M+MP+P_Delta}
 [(M^+-M,P^+-P)]=[\Delta]\in I\otimes\HP^2(R_0). 
\end{align}
Thus the remaining difference
\[
 \Delta^+:=(m_{\mathrm{tar}}-M^+,\pi_{\mathrm{tar}}-P^+)
           =\Delta-L(g,\eta_B)
\]
is an exact degree-two cocycle.

\medskip
\noindent\textbf{Step 4: remove the exact difference and specify the isomorphism.}
\label{splitrev:step4}
We now apply a variant of the coordinate-change
arguments just below \eqref{eq:gauge-differential}, to lift the 
Poisson isomorphism 
$\psi_n:\mathcal E_n
=\mathcal U_A\widehat\otimes_{D_A,\varphi_n}\mathcal B_n
\xrightarrow{\sim}\mathcal D_n$
over $\Lambda_n$ to an isomorphism over $\Lambda_{n+1}$.

Both Poisson structures have the same unit $1$, so their
difference $\Delta^+$ is normalized.
Since $\Delta^+$ is exact, choose a 
(normalized\footnote{it is automatic in this situation though}) $1$-cochain
$u\in I\otimes DP^1(R_0)$ with $Du=\Delta^+$.
In terms of the preceding corrections, this says
\begin{equation}
 \begin{gathered}
 \Delta=L(g,\eta_B)+Du,\qquad D\eta_B=0,\\
 u\in I\otimes\mathrm{DP}^1(R_0)={\rm Hom}_{\C}(R_0,I\otimes R_0),\quad u(1)=0.
 \end{gathered}
 \label{eq:splitrev-correction}
\end{equation}
Extend $u$ to a $\Lambda'$-linear map endomorphism of $R'$. 
It takes values in $IR'$ and vanishes on $IR'$, so 
$u^2$ (the self-composite with itself) vanishes. 
We now 
set a $\Lambda'$-linear selfmap 
\[
 F:=\mathrm{id}-u,\qquad F^{-1}=\mathrm{id}+u,
\]
both mapping $R'$ to itself. We also 
extend $Du$ as follows: 
for each component $c=(Du)_m$ or $c=(Du)_\pi$, define its
extension to $R'$ by the composite
\[
R'\times R'
\xrightarrow{\ \bmod t\ }
R_0\times R_0
\xrightarrow{\ c\ }
I\otimes R_0
\xrightarrow{\ \sim\ }
IR'.
\]
Explicitly, for $x,y\in R'$, 
writing $\bar x=x\bmod t, \bar y=y\bmod t \in R_0$, 
this extension is 
\[
\widetilde c(x,y):=c(\bar x,\bar y)\in IR',
\]
where we use the identification $I\otimes R_0\simeq IR'$.
It is $\Lambda'$-bilinear because $tIR'=0$. 
We continue to write $Du$ for its extended components.
Formula \eqref{eq:gauge-differential} for the Poisson
$\Lambda'$-algebra $(R',M^+,P^+)$, together with $tI=0$, gives
\begin{equation}\label{eq:splitrev-lifted-differential}
\begin{aligned}
(Du)_m(x,y)
&=M^+(u(x),y)+M^+(x,u(y))-u(M^+(x,y)),\\
(Du)_\pi(x,y)
&=P^+(u(x),y)+P^+(x,u(y))-u(P^+(x,y))
\end{aligned}
\end{equation}
for all $x,y\in R'$. 
By \eqref{eq:splitrev-correction} and the definition of
$\Delta^+$,
\[
m_{\mathrm{tar}}-M^+=(Du)_m,
\qquad
\pi_{\mathrm{tar}}-P^+=(Du)_\pi.
\]
Both differences are $IR'$-valued and $\Lambda'$-bilinear.
Thus, since $u(x),u(y)\in IR'$, we have
$m_{\mathrm{tar}}(u(x),y)=M^+(u(x),y), m_{\mathrm{tar}}(x,u(y))=M^+(x,u(y))$ and the same holds for $\pi_{\rm tar}$. 
Similarly, 
$m_{\mathrm{tar}}(u(x),u(y))
=\pi_{\mathrm{tar}}(u(x),u(y))=0$. 
Expanding $F=\mathrm{id}-u$ and using these observations,
we obtain
\[
\begin{aligned}
m_{\mathrm{tar}}(Fx,Fy)
&=m_{\mathrm{tar}}(x,y)
  -m_{\mathrm{tar}}(u(x),y)
  -m_{\mathrm{tar}}(x,u(y))\\
&=M^+(x,y)+(Du)_m(x,y)
  -M^+(u(x),y)-M^+(x,u(y))\\
&=M^+(x,y)-u(M^+(x,y))
 =F(M^+(x,y)),\\[4pt]
\pi_{\mathrm{tar}}(Fx,Fy)
&=\pi_{\mathrm{tar}}(x,y)
  -\pi_{\mathrm{tar}}(u(x),y)
  -\pi_{\mathrm{tar}}(x,u(y))\\
&=P^+(x,y)+(Du)_\pi(x,y)
  -P^+(u(x),y)-P^+(x,u(y))\\
&=P^+(x,y)-u(P^+(x,y))
 =F(P^+(x,y)),
\end{aligned}
\]
for $x,y\in R'$. 
The last line of each calculation follows by substituting
\eqref{eq:splitrev-lifted-differential}, which cancels the
two terms involving $u$ in an input. 
Since $F$ is an invertible $\Lambda'$-linear map with $F(1)=1$,
the above identities show that 
\[
F:(R',M^+,P^+)
\xrightarrow{\ \sim\ }
(R',m_{\mathrm{tar}},\pi_{\mathrm{tar}})
\]
is an isomorphism between (a priori different) Poisson 
$\Lambda'$-algebras. 

The fixed universal-family trivialization gives a Poisson identification
\[
 \chi_{n+1}:\mathcal E_{n+1}:=
 \mathcal U_A\widehat\otimes_{D_A,\varphi_{n+1}}\mathcal B_{n+1}
 \xrightarrow{\ \sim\ }(R',M^+,P^+).
\]
Define
\[
 \psi_{n+1}:=q\circ F\circ\chi_{n+1}:
 \mathcal E_{n+1}\xrightarrow{\ \sim\ }\mathcal D_{n+1}.
\]
All changes are valued in $IR'$ or $IB'$, so that the chosen module
identifications lift the previously fixed ones i.e., 
\[
 \psi_{n+1}\otimes_{\Lambda'}\Lambda_n=\psi_n,\qquad
 f_{j,n+1}\bmod I=f_{j,n}.
\]
This completes the inductive step. 

\medskip
\noindent\textbf{Step 5: keep the actions and take inverse limits.}
\label{splitrev:step5}
Let $H:=T\times\Gamma$ denote the product of the given torus
$T$ and the commuting finite group $\Gamma$ acting on
$\mathcal D$.
The induced action on the marked central fiber
$(A\otimes B,\pi_A+0)$ preserves the two factors
$A\otimes1$ and $1\otimes B$.
Here $T$ includes any torus that moves the deformation
parameter $t$. 
Consider the groupoid of marked flat Poisson deformations
of $(A,\pi_A)$, following \cite{Rim}. 
This groupoid is homogeneous: flat commutative algebras
and their Poisson brackets glue over fibre products of
Artin base rings with one map surjective.
The given $H$-action on $A$ induces an action on this
deformation groupoid.

The tangent space is the finite-dimensional $\HP^2(A)$,
and the infinitesimal automorphisms form the rational
$H$-module $\Der^{\Pois}(A)$. 
Since $H$ is linearly reductive, for every finite-dimensional
rational $H$-module $W$ we have
\[
\begin{aligned}
H^1_{\mathrm{rat}}
  \bigl(H,W\otimes\HP^2(A)\bigr)&=0,\\
H^2_{\mathrm{rat}}
  \bigl(H,W\otimes\Der^{\Pois}(A)\bigr)&=0.
\end{aligned}
\]
These are the cohomological vanishing conditions required
for existence in \cite[p. 225, Theorem]{Rim}.
Indeed, it allows 
us to choose the formal miniversal family
$\mathcal U_A\to\operatorname{Spf}D_A$ with compatible
$H$-actions on both $D_A$ and $\mathcal U_A$, extending
the given action on $A$ and making the family equivariant.
These actions are algebraic at every finite order, and
the bracket has the same character as on $A$.
For the corresponding construction in the universal
Poisson case with a $\C^*$-action scaling the bracket,
see also \cite[Theorem 2.1]{NamikawaNotes}.

Choose compatible equivariant sections modulo powers of
$\mathfrak m_{D_A}$ to lift $\mathfrak m_{D_A}/\mathfrak m_{D_A}^2$.
Formal smoothness makes the lifted basis a system of formal coordinates.
Put
\[
 V_{\mathrm{par}}:=\langle h_1,\ldots,h_s\rangle
                 \simeq\mathfrak m_{D_A}/\mathfrak m_{D_A}^2.
\]
At each inductive step, we choose the lifts $f'_j$
and their corrections $g_j$ so that the maps 
\[
\begin{aligned}
V_{\mathrm{par}}&\longrightarrow tB',
& h_j&\longmapsto f'_j,\\
V_{\mathrm{par}}&\longrightarrow IB',
& h_j&\longmapsto g_j
\end{aligned}
\]
are $H$-equivariant linear maps. 
We next explain how to make these equivariant linear maps 
in the induction on the order of $h_i$s. 
At finite order (modulo $t^{n+1}$, or modulo
$\mathfrak m_{D_A}^{r+1}$ for the universal family),
we make the choices in Steps 1-4 equivariantly in their
original order, using the standard 
Reynolds projection $v\mapsto \mathcal{R}(v)$ 
for various $H$-modules and their elements $v$. 
It is nothing but the average operator with respect to (any) 
compact real form $H_c$ of $H$. 
For module maps, averaging is performed in the affine space of
unit-preserving linear lifts of a fixed equivariant reduction.
An averaged lift remains invertible whenever that reduction is an
isomorphism, since comparison with any invertible lift gives the identity
plus a nilpotent endomorphism.

First choose the universal-family module trivialization
equivariantly by applying $\mathcal R$ at successive
$\mathfrak m_{D_A}$-adic orders.
In Step 1, apply $\mathcal R$ to various isomorphisms between $H$-representations, the tentative operations on $B'$, and the lift $V_{\mathrm{par}}\to tB'$. 
In Step 2, apply it to $\ell$ and then to the normalized
corrections $(\epsilon_m,\epsilon_\pi)$. 
In Step 3, first choose $q$ equivariantly, 
take $(g,\eta_B,u)$ as $H_c$-invariant elements 
(by the Reynolds operators) and apply the same arguments. 

Finally, in Step 4 use $u^{\mathrm{eq}}:=\mathcal R(u)$ to form
\[
 F=\mathrm{id}-u^{\mathrm{eq}},
 \qquad
 \psi_{n+1}=q\circ F\circ\chi_{n+1}.
\]
Note that $\psi_{n+1}$ is then $H$-equivariant. 
Finally take the compatible inverse limits
\[
 \mathcal B:=\varprojlim_n\mathcal B_n,\qquad
 f_j:=\varprojlim_n f_{j,n},\qquad
 \psi:=\varprojlim_n\psi_n.
\]
The compatible module identifications give
$\mathcal B\simeq B[[t]]$ as $\mathbb C[[t]]$-modules; 
Taking inverse limits therefore yields a marked 
equivariant Poisson isomorphism  
\[
 \psi:\mathcal U_A\widehat\otimes_{D_A,\varphi}\mathcal B
       \xrightarrow{\ \sim\ }\mathcal D,\qquad
 \varphi(h_j)=f_j.
\]
\end{proof}

\medskip
\noindent
Our proof can be summarized as follows. 
Step 2 corrects $B'$ using the vanishing of $[\rho_\ell(\mathrm{ob})]\in I\otimes\HP^3(B,0)$ and establishes that the $f'_j$ are Casimir 
by analyzing the mixed term of the obstruction. Then Step 3 uses 
 $\Delta$, which is the difference of $(M,P)$ with the target operations
pulled back along temporary comparison $q$ with $\mathcal{D}_{n+1}$, 
to modify to obtain $(M^+,P^+)$. Summarizing up, we have the following 
picture: 
\[
\begin{tikzcd}[column sep=2.5em,row sep=2em]
(M_{\rm pre},P_{\rm pre})
  \arrow[r,dashed,"\text{in Step 2}"]
& (M,P)
  \arrow[r,dashed,"\text{in Step 3}"]
& (M^+,P^+)
  \arrow[d,"F=\id-u","\simeq"']
\\
&& (m_{\mathrm{tar}},\pi_{\mathrm{tar}})
  \arrow[r,"q","\simeq"']
&[-1.6em] \mathcal D_{n+1}.
\end{tikzcd}
\]
The remaining exact difference $\Delta^+=Du$ is then removed by
the coordinate change $F=\id-u$, where
$u\in I\otimes\DP^1(R_0)$ is a linear cochain with $u(1)=0$ obtained in Step4. 

\section{Symplectic product cover and formal rigidity}\label{sec:blocks}\label{sec:formal}
This section is the core of the rigidity argument, 
combining the 
preparatory results of Sections \ref{sec:analysis} and \ref{sec:PD}.
It follows a single family from the structure of its central fiber to
formal rigidity. Except for \S \ref{sec:intermediate},
which verifies \ref{hyp:form} in the symplectic singularity application, 
the discussions below apply to any family satisfying
\ref{hyp:grading}-\ref{hyp:metric}. 
The limit of the given symplectic form can be degenerate, but in this 
particular situation, we can and do apply Theorem \ref{thm:SPD} to obtain a
formal product on its cover and induced deformations which we observe to be 
rigid in an inductive manner (it can not be generalized too much; see Remark 
\ref{remark:productcounterexample}). That is the core idea. 

\subsection{On the homogeneous holomorphic symplectic 
form}\label{sec:intermediate}
This subsection supplies \ref{hyp:form} for the algebraic conification application, independently of the conclusion $W=C$. For a normalized-volume minimizing valuation $v$,
\[
W=\Spec\gr_v\OO_{X,x},\qquad \gr_v\OO_{X,x}=\bigoplus_\lambda\mathfrak a_\lambda(v)/\mathfrak a_{>\lambda}(v).
\]
\begin{proposition}\label{prop:formW}
Let $(W,\xi)$ be a K-semistable symplectic affine cone of dimension $n=2m$, normalized by $A_W(\xi)=2m$. For the (complexified) torus $T$ generated by the Reeb vector $\xi$, $W$ admits a $T$-homogeneous 
algebraic 
symplectic form of character $\tau$ with $\tau(\xi)=2$. 
\end{proposition}
Recall that \cite[Theorem 4.5, Theorem 6.16]{NO} partially proved this in many cases.
\begin{proof}
Via the $T$-action,
we decompose an {\it algebraic} symplectic form into its $T$-eigen  components, $\sigma=\sum_\lambda\sigma_\lambda$ where $\lambda$
means the $\xi$-weight for the (volume-minimizing) Reeb vector field $\xi$. By \cite[Theorem 5.10(2)]{XZ} (cf., also related \cite[Theorem 4.5]{NO}), we have $\lambda\ge2$ for every nonzero component. The canonical frame $\sigma^m$ is an eigenvector: a torus translate differs by a unit, and units on a positively graded integral cone are constant. Its weight is $A_W(\xi)=2m$. The weight-$2m$ part of the expansion is therefore
\[
\sigma_2^m=\sigma^m.
\]
Thus $\sigma_2$ is non-degenerate. Differentiation preserves weight, so it is closed, and reflexive-form extension \cite[Theorem 1.4]{GKKP} extends it to a resolution. A Reeb weight determines at most one character of $T_\xi$ by its definition as the torus closure. Hence $\sigma_2$ has the required character.
\end{proof}
For a larger prescribed torus, the single-character condition remains part of
\ref{hyp:form}; the standard application uses $T=T_\xi$.
The analytic comparison of germs is completed in \S \ref{sec:conclusion}.
From \S \ref{sec:flag} onward, $\sigma_W$ is simply the form given
by \ref{hyp:form}; no further use is made of its construction from
$(X,x)$.

\subsection{From successive leading forms to parallel tangent subbundles}\label{sec:flag}
Let $\mathcal U\subset\cW$ be the relative smooth locus and $\mathcal E=T_{\mathcal U/\A^1}$. Since $p$ is flat with normal fibers, $\mathcal U\cap C=C^{\reg}$. Through the
projection from
$\mathcal{W}\setminus C\simeq W\times (\A^1\setminus 0)$ to $W$ and restriction,
we extend
$\sigma_W$ to a rational section on
$\bigwedge^2\mathcal E^*$. Take
the least integer $a$ such that
\[
S=t^a\sigma_W
\]
is in $\Gamma(\mathcal U,\bigwedge^2\mathcal
E^*)$, so that it is a relatively closed form
with the restriction to the central fiber is
$S_0\ne0$. Its $T$-character is $\tau$, and it is
$G_{\test}$-semi-invariant.

At the generic point $\eta_C$ of $C\subset \mathcal{W}$, we take the stalk
$\mathcal{O}_{\mathcal{U},\eta_C}$ as a
discrete valuation ring $R$  with uniformizer $t$. The
elementary divisors of an alternating form give a
local basis  $v_1,w_1,\ldots,v_m,w_m$
of $\mathcal{E}|_{\eta_C}$
in which
\begin{equation}\label{eq:smith}
S=\sum_{j=1}^m t^{q_j}v_j^*\wedge w_j^*,\qquad
0=q_1\le\cdots\le q_m=\qmax.
\end{equation}
The following intrinsic construction shows that
they describe subbundles everywhere on
$C^{\reg}$.
For integers $r\ge0$, define coherent subsheaves on $\mathcal U$
and $C^{\reg}$, respectively, by
\begin{equation}\label{eq:flag}
\begin{aligned}
N_r&:=\{v\in\mathcal E\mid S(v,\mathcal E)\subset t^r\OO_{\mathcal U}\},\\
K_r&:=\im(N_r\longrightarrow\mathcal E|_{C^{\reg}}).
\end{aligned}
\end{equation}
At this stage these are only coherent subsheaves, not assumed locally free;
in particular, the assertion that $K_r\subset T_{C^{\reg}}$ is a subbundle
will be proved in Proposition \ref{prop:blocks}.
In particular $K_0=T_{C^{\reg}}$. Define a 
$\mathcal{O}_{C^{\rm reg}}$-bilinear ($\mathcal{O}_{C^{\rm reg}}$-
valued) pairing on $K_r$ as follows: for any local sections 
$v_0,w_0$ of $K_r$, set 
\begin{equation}\label{eq:leading}
b_r(v_0,w_0):=t^{-r}S(v,w)\bmod t,
\qquad v,w\in N_r\text{ lifting }v_0,w_0.
\end{equation}
This is independent of the lifts: replacing $v$ by $v+tz$ changes the expression by $t^{1-r}S(z,w)$, which is divisible by $t$ since $w\in N_r$. The same holds for the other lift.

\begin{proposition}\label{prop:blocks}
The $K_r$ are parallel algebraic subbundles of $T_{C^{\reg}}$, and
\begin{equation}\label{eq:Kr+1_as_ker,Kvan}
K_{r+1}={\rm Ker}(b_r:K_r\to K_r^*),\qquad K_{\qmax+1}=0.
\end{equation}
For the distinct orders $0=e_1<\cdots<e_k=\qmax$, there is a splitting preserved by both tori,
\[
T_{C^{\reg}}=\bigoplus_{i=1}^kF_i,\qquad
K_r=\bigoplus_{e_i\ge r}F_i.
\]
The non-degenerate pairing induced by $b_{e_i}$ on $K_{e_i}/K_{e_i+1}$ defines a parallel algebraic form $\sigma_i$ supported on $F_i$. Each $\sigma_i$ has $T$-character $\tau$ and is $G_{\test}$-semi-invariant, and
\begin{equation}\label{sigmaCC}
\sigma_C=\sum_{i=1}^k\sigma_i
\end{equation}
is a $g$-parallel algebraic symplectic form.
In particular, $(C^{\rm reg},g)$ is a hyperK\"ahler cone. 
\end{proposition}
\begin{proof}
For any closed point $x$ of $C^{\reg}$, in its neighborhood 
$\mathcal{U}_x\subset \mathcal{U}$, we analyze the structures of 
$K_r$ inductively. Note the generic point 
description \eqref{eq:smith} a priori 
does not suffice to work over whole $C^{\rm reg}$. 

For any closed point $x$ of $C^{\reg}$, in its neighborhood $\mathcal{U}_x\subset \mathcal{U}$, take a local 
$S$-orthogonal decomposition as (free) $\mathcal{O}_{\mathcal{U}_x}$-modules 
\begin{equation}\label{eq:local-flag-normalform}
 \mathcal E=\bigoplus_{0\le a<r}\mathcal H_a\oplus\mathcal E_r,
 \qquad \text{ with }
 S=\bigoplus_{0\le a<r}t^aS_a\oplus t^r S_r,
\end{equation}
where $S_a$ is a non-degenerate 
alternating form on $\mathcal H_a$ and 
$S_r$ is another 
alternating form on $\mathcal E_r$. From the definitions  \eqref{eq:flag} and \eqref{eq:leading}, the decomposition identifies
\begin{equation}\label{eq:KrEr}
 K_r=\mathcal E_r/t\mathcal E_r,\qquad b_r=S_r\bmod t.
\end{equation}
Assume that $K_r$ is a parallel algebraic subbundle of 
$\mathcal{T}_{C^{\rm reg}}$. 
We prove the above by induction on $r$; assuming 
all the above statements \eqref{eq:local-flag-normalform}, 
\eqref{eq:KrEr} and the parallelness of $K_r$ 
hold for fixed $r$ (as it is clearly the case for $r=0$) 
and show that same holds for $r+1$. 

Firstly, the orthogonal projector $P_r$ of $T_{C^{\rm reg}}$ 
to $K_r$ 
is parallel,  hence holomorphic, and is invariant under the two 
compact tori. It is therefore invariant under the complex tori 
which implies the algebraicity. 
Consequently
\[
 \widetilde b_r(v,w)=b_r(P_rv,P_rw)
 \in H^0(C,\Omega_C^{[2]})_\tau,
 \qquad \langle\tau,\xi\rangle=2.
\]
Theorem \ref{thm:bochner} makes $\widetilde b_r$ parallel. Thus $\ker b_r$ has constant rank and is a parallel algebraic subbundle of $K_r$.

On small enough 
neighborhoods of $x$, 
choose an arbitrary complement bundle 
to $\ker b_r$ in $\mathcal E_r/t\mathcal E_r$ 
as 
$\overline{\mathcal H}_r$, and lift it to a direct summand $\mathcal H_r\subset\mathcal E_r$. 
From the construction, $S_r$ restricted to 
$\mathcal H_r$ is non-degenerate. 
Thus, 
\[
 \mathcal E_r=\mathcal H_r\oplus\mathcal E_{r+1},
 \qquad
 \mathcal E_{r+1}/t\mathcal E_{r+1}
 =\overline{\mathcal H}_r^{\perp_{b_r}}=\ker b_r.
\]
Hence $S_r|_{\mathcal E_{r+1}}\equiv0\pmod t$. Since $t$ is a
non-zero-divisor, the alternating form
$S_{r+1}:=t^{-1}S_r|_{\mathcal E_{r+1}}$ is regular, and
\[
 S=\left(\bigoplus_{0\le a<r}t^aS_a\right)
   \oplus t^r(S_r|_{\mathcal H_r})\oplus t^{r+1}S_{r+1}.
\]
This proves \eqref{eq:local-flag-normalform} at the next stage.
Moreover, \eqref{eq:flag} gives
\[
 N_{r+1}=\left(\bigoplus_{0\le a<r}t^{r+1-a}\mathcal H_a\right)
          \oplus t\mathcal H_r\oplus\mathcal E_{r+1},
\]
so taking the image modulo $t$ yields
$K_{r+1}=\mathcal E_{r+1}/t\mathcal E_{r+1}$ by its definition 
which then coincides with $\ker b_r$ for $b_{r+1}:=S_{r+1}\bmod t$. 

The local summands a priori may not glue, but the equality 
$K_{r+1}=\ker b_r$ holds globally since both sides are 
intrinsically 
defined sheaves. Since $\ker b_r$ is parallel and
algebraic by the preceding Bochner step, 
we complete the induction. 
In particular, the former half of \eqref{eq:Kr+1_as_ker,Kvan} 
holds. 
Further, combined 
with \eqref{eq:smith}, $K_{\qmax+1}=0.$ 

On $C^{\reg}$, let $P_r$ be the $g$-orthogonal projector (morphism) 
from 
$T_{C^{\rm reg}}$ onto $K_r$ and set 
\[
 \pr_i:=P_{e_i}-P_{e_i+1},
 \qquad F_i:={\rm Im}(\pr_i) 
       =K_{e_i}\cap K_{e_i+1}^{\perp_g}.
\]
These projectors are parallel, algebraic, and equivariant for both tori, and
\[
 T_{C^{\reg}}=\bigoplus_iF_i,
 \qquad K_r=\bigoplus_{e_i\ge r}F_i,
 \qquad F_i\xrightarrow{\sim}K_{e_i}/K_{e_i+1}.
\]
Since $\ker b_{e_i}=K_{e_i+1}$, the restriction $b_{e_i}|_{F_i}$
is nondegenerate. Define
\[
 \sigma_i(v,w):=b_{e_i}(\pr_i v,\pr_i w).
\]
Then
\[
 \sigma_i|_{F_i}=b_{e_i}|_{F_i},
 \qquad \ker\sigma_i=\bigoplus_{j\ne i}F_j,
 \qquad \nabla\sigma_i=0.
\]
Since $S$ is $T$-eigenvector with character $\tau$, $N_r, K_r$ 
and hence $F_i$ are all $T$-invariant. $\sigma_i$ is again 
$T$-eigen with the character $\tau$ and the same holds for 
the sum 
$\sigma_C:=\sum_i\sigma_i$. 
Since $\sigma_C$ is nondegenerate and parallel, it is also 
$d$-closed. 
These algebraic forms define reflexive forms on $C$;
\cite[Theorem 1.4]{GKKP} extends them regularly to any resolution. 
Because of the $g$-parallelness of the holomorphic symplectic form 
$\sigma_C$ on $C^{\rm reg}$, 
the last hyperK\"ahler assertion also holds. 
\end{proof}

\begin{proposition}\label{prop:regularPi}
The relative tensor
\[
\Pi:=t^{\qmax}S^{-1}\in\Gamma(\mathcal U,\textstyle\bigwedge^2T_{\mathcal U/\A^1_t})
\]
defines a regular Poisson bracket on $\cR$. Its restriction to $C$ is $\sigma_k^{-1}\oplus0$ in the above splitting. Its $T$-character is $-\tau$, and it is $G_{\test}$-semi-invariant. 
\end{proposition}
\begin{proof}
The largest pole order of $S^{-1}$ along $C$ is $\qmax$ by \eqref{eq:smith}, and there are no horizontal poles. Hence $\Pi$ is regular on $\mathcal U$. Its values on regular functions extend across the complement, of codimension at least two, by normality. Leibniz and Jacobi hold on the punctured family and hence everywhere. The local decomposition in Proposition \ref{prop:blocks} shows that, 
after multiplying $S^{-1}$ by $t^{\qmax}$ 
and restricting to the central fiber $C^{\reg}$, only the summand corresponding 
to $e_i=\qmax$ survives. Hence $\Pi_0=\sigma_k^{-1}\oplus0$. 
\end{proof}

With $X_f=\{f,-\}=\Pi^\sharp(df)$ and $S^\flat(v)=\iota_vS$, our inverse convention gives
\begin{equation}\label{eq:inverse}
S^\flat\Pi^\sharp=t^{\qmax}\id.
\end{equation}
Thus a regular bracket is available even though $S$ itself may be degenerate on the central fiber. 

\subsection{Algebraically integrable leaves}\label{sec:leaves}
Note that each $F_i$ is a foliation, since it is the kernel of the closed form $\alpha_i=\sum_{j\ne i}\sigma_j$. Assume $k\ge2$; when $k=1$ the product cover below is the identity. This subsection 
aims to show the following. 
\begin{lemma}\label{lem:leaves}
The general leaves of $F_i$ are algebraically integrable.
\end{lemma}
\begin{proof}
We use an auxiliary algebraic quotient by $\C^*$. 
Choose a positive rational Reeb direction $\upsilon\in N\otimes_\Z\Q$ sufficiently close to $\xi$ that $\tau(\upsilon)>0$, and a positive integral multiple $\lambda\in
N\cap\Q_{>0}\upsilon$ defining a quotient action. The resulting
quotient (Deligne-Mumford) stack is denoted as
$\mathcal{X}_{\upsilon}$. Its coarse space
$(X_\upsilon,\Delta_\upsilon)$ is log Fano with the natural boundary $\Delta$ with standard coefficient:
let $e$ be the multiplicity of the Seifert bundle $q\colon
C^{\times}\to X_\upsilon$
structure along
the boundary component, so its coefficient in $\Delta_\upsilon$ is $1-
1/e$.
It need not be K-semistable. The quotient has dimension $n-1$; put $H=-(K_{X_\upsilon}+\Delta_\upsilon)$.
Let
\[
q:C^\times:=C\setminus\{0\}\longrightarrow X_\upsilon
\]
denote the auxiliary $\Gm$-quotient defined by $\lambda$.
Following \cite[\S5.1]{XZ}, choose a finite Galois cover
$f:Y_\upsilon\to X_\upsilon$ adapted to $(X_\upsilon,\Delta_\upsilon)$,
meaning that $f^*\Delta_\upsilon$ has integral coefficients.
Such a cover exists by \cite[Definition 2.37 and Proposition 2.38]{CKT21}.
Write $\Omega^{[1]}_{(X_\upsilon,\Delta_\upsilon,f)}$ for the adapted
reflexive cotangent sheaf on $Y_\upsilon$, as in
\cite[Definition 3.11]{CKT21} and \cite[Definition 5.2]{XZ}.

Let $M$ be the ramification index of the chosen adapted cover;
adaptedness gives $e\mid M$. Then, the
defining equations of the boundary components are written generically by $x, w$ on $X_\upsilon, Y_\upsilon$ respectively with
$x=w^M$. We have the boundary
coefficient $1-1/e$ downstairs and $e\mid M$, so it is generated by
$w^{M/e-1}dw$ and the differentials of the remaining coordinates;
the resulting locally free sheaf extends reflexively to
$Y_\upsilon$.
Its dual
\[
T_{(X_\upsilon,\Delta_\upsilon,f)}
:=\bigl(\Omega^{[1]}_{(X_\upsilon,\Delta_\upsilon,f)}\bigr)^\vee
\]
is the adapted tangent sheaf on $Y_\upsilon$; see also
\cite[Definition 4.1 and Remark 4.2]{CKT21} and
\cite[Definition 2.23]{GT22}.
Here $T_{X_\upsilon,\mathrm{orb}}$ denotes the orbifold tangent
sheaf of the quotient pair, interpreted through these adapted covers
as in \cite[Definition 5.7 and Remark 5.4]{XZ}.

On $C^\times$, its infinitesimal generator $Z_\lambda$ spans
the vertical tangent direction, and one has the exact sequence
\[
0\longrightarrow
\OO_{C^\times}\cdot Z_\lambda
\longrightarrow
T_{C^\times}
\longrightarrow
\mathcal{G}
\longrightarrow0,
\]
where $\mathcal{G}$ can be written as
$q^*T_{\mathcal{X}_\upsilon}$ for the stack $\mathcal{X}_\upsilon$.
Descending this sequence in the orbifold charts and pulling it
back to $Y_\upsilon$ gives the canonical tangent extension
(cf., \cite[\S2.3, Proposition 2.10]{LiExt}), extended
reflexively as in
\cite[Definition 5.3 and Remark 5.4]{XZ}. The following sequence is exact on
a big open subset of $Y_\upsilon$ containing all codimension-one points; its
terms and maps are then extended reflexively, with the same notation:

\begin{align}\label{eq:canonical-extension}
0\longrightarrow
\OO_{Y_\upsilon}
\longrightarrow
\mathcal A
\longrightarrow
T_{(X_\upsilon,\Delta_\upsilon,f)}
\longrightarrow0.
\end{align}
All terms in \eqref{eq:canonical-extension} are sheaves on
$Y_\upsilon$. 
The first term is the descent of the vertical
$Z_\lambda$-direction. Let $L_\upsilon$ be the ample
$\Q$-polarization defining the Seifert bundle $q$. 
Since $H$ is a positive rational multiple of $L_\upsilon$, the
extension is isomorphic to the dual of
$E_{(X_\upsilon,\Delta_\upsilon,f)}$, the canonical cotangent
extension with class $f^*c_1(K_{X_\upsilon}+\Delta_\upsilon)$
\cite[Definition 5.3]{XZ}.
For a torsion-free sheaf, let $\mu_{\min,H}$ denote the infimum of slopes of positive-rank torsion-free quotients. Dualizing \cite[Lemma 5.9]{XZ} gives
\begin{equation}\label{eq:slope}
\mu_{\min,f^*H}(\mathcal A)\ge\frac{\min\{1,\delta(X_\upsilon,\Delta_\upsilon)\}(f^*H)^{n-1}}{n}>0.
\end{equation}
For the last inequality, we only need general positivity of 
the delta invariant
$\delta(X_\upsilon,\Delta_\upsilon)$ of a log Fano pair, 
introduced and shown in \cite{FujitaOdaka, BlumJonsson}. 

The splitting $T_{C^{\reg}}=\bigoplus_i F_i$ is invariant
under the quotient $\mathbb G_m$-action.
Hence its projectors descend to idempotent endomorphisms
of the Atiyah bundle on the orbifold quotient stack. 
These projectors pull back to 
$\mathcal A$ on $Y_\upsilon$: in codimension one, the 
local orbifold and adapted-cover coordinates satisfy 
$x=z^e=w^M$ with $e\mid M$, 
so the pullback $z=w^{M/e}$ preserves regularity
of the projectors in the orbifold bundle charts
(cf.\ \cite[Remark 5.4]{XZ}). 
They extend uniquely by reflexivity, giving 
\[
\mathcal A=\bigoplus_i\mathcal A_i,
\qquad
\mathcal A_i:=\operatorname{Im}(\overline P_i). 
\]
Summarizing up, the splitting $T_{C^{\rm reg}}=\oplus_i F_i$
descends to direct summands 
$\mathcal A_i$ of $\mathcal{A}$.
Their projections through \eqref{eq:canonical-extension} are generically injective: otherwise $Z_\lambda$ would lie in $F_i={\rm Ker}\alpha_i$, contradicting $d\iota_{Z_\lambda}\alpha_i=L_{Z_\lambda}\alpha_i=\tau(\lambda)\alpha_i\ne0.$
Let $\mathcal F_i\subset T_{X_\upsilon}$ be the saturated projected foliation. For local coordinates
$x$ and $w$ of $X_\upsilon$ and $Y_\upsilon$ respectively,
along the branch divisors we have
$x=w^M$ for the cover $f$
with boundary coefficient $1-1/e$ of $\Delta_\upsilon$, its normal
generator maps from $T_{(X_\upsilon,\Delta_\upsilon,f)}$
to $f^{[*]}T_{X_\upsilon}$
as
\[
w^{1-M/e}\partial_w\longmapsto Mw^{M-M/e}\partial_x.
\]
Combining to the natural projection $\mathcal{A}_i$ to $T_{(X_\upsilon,\Delta_\upsilon,f)}$, we obtain a natural
morphism $\mathcal{A}_i\to f^{[*]}T_{X_\upsilon}$.

For every positive-rank torsion-free quotient $Q$ of $\mathcal F_i$, the image
\[
 \mathcal I_Q=\im\bigl(\mathcal A_i\longrightarrow f^{[*]}Q\bigr)
\]
is a quotient of $\mathcal A_i$ of the same rank as $Q$. Here $f^{[*]}Q=(f^*Q/\mathrm{torsion})^{**}$; maps are first constructed in codimension one and extended to this reflexive target. Reflexive extension does not change slopes. Let $c>0$ denote the lower bound for $\mu_{\min,f^*H}(\mathcal A)$ in \eqref{eq:slope}. Since $\mathcal A_i$ is a direct summand, it has the same lower bound on quotient slopes. The inclusion of equal-rank sheaves gives
\[
 \deg(f)\,\mu_H(Q)\ge\mu_{f^*H}(\mathcal I_Q)\ge c>0.
\]
Thus $\mu_{\min,H}(\mathcal F_i)\ge c/\deg(f)>0$, uniformly over $Q$. By the Mehta-Ramanathan-Flenner restriction theorem in the characteristic-zero form of \cite{Flenner}, applied to the Harder-Narasimhan filtration, the restriction of $\mathcal F_i$ to a sufficiently general complete-intersection curve cut out by $n-2$ divisors in $|mH|$ is ample; see \cite[Theorem 2.4]{Hartshorneample}. The curve avoids the singular loci of the quotient and that of the foliation. The Bogomolov-McQuillan criterion \cite{BM} in \cite[Theorem 1]{KST} applies directly on a normal projective variety. Such complete-intersection 
curves sweep out whole $X_v$, so the criterion proves algebraic integrability of $\mathcal F_i$. 

Now we want to lift this conclusion from the quotient to the cone
$C$ (note that the $Z_\lambda$-direction is not contained in
any of $F_i$, as far as $\#\{i\}>1$, and the preimage of
the integrated foliation of $\mathcal{F}_i$ is too big). 
Let us fix $i$.
Put $r_i=\rk F_i$, $w=\tau(\lambda)>0$, and set
\begin{equation}\label{eq:leaf-primitive}
\theta_i:=\frac1w\iota_{Z_\lambda}\alpha_i.
\end{equation}
Then $d\theta_i=\alpha_i$ follows from $d\alpha_i=0$ and
$L_{Z_\lambda}\alpha_i=w\alpha_i$. Set $s=n-1-r_i$. A rational
first integral of the projected
foliation $\mathcal F_i$ means a rational function
$H\in\C(X_\upsilon)$ satisfying $u(H)=0$ for every rational vector
field $u$ tangent to $\mathcal F_i$. The simple fact $\iota_{Z_\lambda}\theta_i=0$ plays a key role below.

Since $\mathcal F_i$ is algebraically integrable, there is a
dominant
rational leaf-space map $p_i:X_\upsilon\dashrightarrow B$ with
$\dim B=s$, whose general fibers have the leaves as dense open
subsets. Choose a transcendence basis of $\C(B)$ over $\C$ and let
$H_1,\ldots,H_s\in\C(X_\upsilon)$ be its pullbacks.
These are independent rational first integrals: here independence
means algebraic independence over $\C$, or equivalently, in
characteristic zero, linear independence of $dH_1,\ldots,dH_s$
at the generic point. Put $h_j:=q^*H_j\in\C(C).$
Their general common level sets on $C$ have dimension
$n-s=r_i+1$, hence we need
one more function (first integral)
to cut against the redundant $\Gm$-direction.

At the generic point of $C$, the independent forms $dh_1,\ldots,dh_s$
form a basis of the annihilator of
$F_i+\C(C)Z_\lambda$, since this annihilator has dimension $s$.
For the desired algebraic integrability of $F_i$,
we want to construct additional rational function, which will be denoted as $g_1$, so that
$dh_1,\ldots,dh_s,dg_1$ span the annihilator of
$F_i$, from $\alpha_i$ and its interior derivative
$\theta_i$.

As $\theta_i$ annihilates both $F_i$ and $Z_\lambda$, there are
unique
rational functions
$g_j\in\C(C)$ such that
\begin{equation}\label{eq:leaf-integrals}
\theta_i=\sum_{j=1}^s g_j\,dh_j.
\end{equation}
Now, we want to see the behaviour of $g_j$ along the $Z_\lambda$-
direction.
For $v\in F_i$, from the definition of $\theta_i$ i.e., \eqref{eq:leaf-primitive} and equations below
give $L_v\theta_i=\iota_vd\theta_i+d\iota_v\theta_i=0$.
Since $v(h_j)=0$, differentiating \eqref{eq:leaf-integrals} yields
$\sum_jv(g_j)\,dh_j=0$, hence $v(g_j)=0$ i.e., $g_j$
is constant along the $F_i$-leaves.
Similarly, $\theta_i$ is also homogeneous of $Z_\lambda$-weight $w$ i.e., $L_{Z_\lambda}\theta_i=w\theta_i$. Using
$Z_\lambda h_j=0$, we can also confirm $g_j$ also has the same
homogeneity i.e., $Z_\lambda g_j=w g_j.$

In summary,
each $g_j$ is also constant along the $F_i$-leaves, whereas
it has weight $w$ in the quotient-action direction.
Since $Z_\lambda\notin F_i=\ker\alpha_i$ as we previously explained, one has $\theta_i\ne0$
so some coefficient, say $g_1$, is nonzero. Moreover,
\[
\iota_{Z_\lambda}dg_1=w g_1\ne0,
\qquad \iota_{Z_\lambda}dh_j=0.
\]
Consequently $dg_1$ is independent of $dh_1,\ldots,dh_s$.
The rational map
\[
\Phi=(h_1,\ldots,h_s,g_1):C\dashrightarrow\A^{s+1}
\]
has generic differential of rank $s+1=n-r_i$, with kernel $F_i$.
On a suitable dense open subset, the connected components of its
general fibers are therefore leaves of the restricted foliation.
Their closures in $C$ coincide with the closures of the corresponding
general leaves and have dimension $r_i$, proving generic
algebraic integrability of $F_i$.
\end{proof}

\subsection{The finite product cover}\label{sec:cover}
Now we want to apply the technique of 
\cite[Theorem 25]{DGP} of somewhat deRham 
decomposition type result, to get some 
product structure using  
the above algebraic foliations 
$F_i$s (Lemma \ref{lem:leaves}). We add some extra arguments to 
refine to the following, in our symplectic cone setup. 

Write $C=\Spec A$, which we assume to be not 
smooth (otherwise the main assertions are trivial) and  $U=C^{\reg}\setminus\{o\}$. Keep the notations from 
Proposition \ref{prop:blocks} and Lemma \ref{lem:leaves}. Put
$n_i=\rk_\C F_i$, which is even. The original tori $T$ and $G_{\test}$ 
preserve each summand $F_i$ and its projector. 
Each $T$-character space $A_\chi$ is finite-dimensional. Moreover, only
finitely many Reeb weights occur in a bounded interval: finitely many
positive-weight homogeneous generators have only finitely many monomials
of bounded weight. Thus we may exponentiate character-preserving derivations
and induct on the original real Reeb weights, without changing $T$ or $\xi$.

\begin{theorem}[Finite cover as product]\label{thm:cover}
There is a finite Galois morphism
\[
 \phi:V=\prod_{i=1}^kV_i\longrightarrow C
\]
from a product of normal affine conical symplectic varieties $(V_i,\omega_i)$ such that
$\phi^*\sigma_i=\pr_i^*\omega_i$ on $V^{\reg}$, 
which integrates the pullback of each $F_i$. 
The deck group $\Gamma$ acts on each $V_i$ 
by symplectic automorphisms 
with respect to $\omega_i$. 

Further, $\phi$ is \'etale outside (complex) codimension $4$. 
For some finite isogenies $\alpha:T'\to T$ and 
$\beta:G'_{\test}\to G_{\test}$, both $T'$ and $G'_{\test}$ acts on 
$V$ commutatively (also commuting with $\Gamma$-action) 
such that the former action is good, 
$\phi$ is $G'_{\test}\times T'$-equivariant, and preserve 
each $(V_i,\C\omega_i)$. 
\end{theorem}

In particular, \eqref{sigmaCC} pulls back to the product symplectic form
\[
 \phi^*\sigma_C=\sum_{i=1}^k\pr_i^*\omega_i
 \qquad\text{on }V^{\reg},
\]
where the pullback is understood by its unique extension from
$\phi^{-1}(C^{\reg})$.

This is clearly analogous to \cite[Theorem 25]{DGP} 
but comparing with their original statements, 
we note the following technical 
differences: firstly, $C$ is non-compact, 
the obtained cover is proven to be Galois, 
its factors are conical symplectic varieties and the 
non-\'etale locus has codimension at least $4$. 

\begin{proof}
We use $F_i$s and their algebraic leaves (which exist by
Lemma \ref{lem:leaves}) as key inputs. Throughout this proof,
\eqref{eq:radius-comparison} gives
$\kappa^{-1}r_\xi\le r\le\kappa r_\xi$ on $U$, where $r_\xi$ is continuous
and proper and $\kappa\ge1$.
First, let $P_i$ be the orthogonal projector from $T_U$ onto 
$F_i|_U$, also acting on the corresponding real 
$I$-invariant tangent distribution. Put 
\begin{equation}\label{eq:mr-operators}
 E_i=P_iE,\quad R_i=P_iR,\quad h_i=|R_i|^2,\qquad
 D_i=E-E_i,\quad S_i=D_i+\overline{D_i}=R-R_i.
\end{equation}
Algebraicity and normality make $E_i,D_i$ derivations of $A$.
Invariance of $P_i$ makes them commute with both original tori, hence
preserve each $A_\chi$. Parallelism and $\nabla R=\id$ give
\[
 \nabla R_i=P_i,\qquad dh_i=2g(R_i,-),\qquad \sum_i h_i=r^2.
\]
Since $S_i$ is orthogonal to $R_i$, these imply
\begin{equation}\label{eq:mr-ODE}
 S_i h_i=0,\qquad S_i(r^2-h_i)=2(r^2-h_i),\qquad (IS_i)r^2=0.
\end{equation}
For every $z\in\C$, matrix exponentiation on the finite-dimensional
spaces $A_\chi$ defines $e^{zD_i}$ on $A$. The derivation rule implies
that these maps preserve multiplication, and $e^{-zD_i}$ is their
inverse. They therefore define a holomorphic flow
$\Psi_{i,z}\colon C\to C$ whose maps are algebraic automorphisms,
with $\Psi_{i,z}^*=e^{zD_i}$. The flow fixes $o$ and preserves $U$.
Its real and imaginary generators are $S_i$ and $IS_i$.
\eqref{eq:mr-ODE} implies that, for $x\in U$ and $t\in\R$,

\begin{equation}\label{eq:mr-complement-flow}
 r^2(\Psi_{i,t}x)=h_i(x)+e^{2t}\bigl(r^2(x)-h_i(x)\bigr).
\end{equation}
On $U$, the imaginary-time flow preserves $r$, while the negative
real-time flow does not increase it. Put
\[
 K_U:=\{x\in U:r(x)\le1\},\qquad K:=\overline{K_U}\subset C.
\]
The comparison with $r_\xi$ gives $K\subset\{r_\xi\le\kappa\}$, so $K$
is compact. Smoothness and homogeneity of $r$ on $U$ show that $K$
contains a nonempty open subset. The flow properties on $K_U$ give
$\Psi_{i,it}(K_U)=K_U$ for $t\in\R$ and
$\Psi_{i,t}(K_U)\subset K_U$ for $t\le0$. Continuity of the flow maps
extends these same properties to $K$.
On $A_\chi$ use the norm $\|f\|_K=\sup_K|f|$. We have
\begin{equation}\label{eq:mr-norm}
 \|e^{itD_i}f\|_K=\|f\|_K\quad(t\in\R),\qquad
 \|e^{tD_i}f\|_K\leq\|f\|_K\quad(t\leq0).
\end{equation}
The first equality excludes nonreal eigenvalues and nontrivial Jordan
blocks of $D_i|_{A_\chi}$, since either would give an unbounded
imaginary-time flow. The second gives $e^{tb}\le1$ for an eigenvalue $b$
and $t\le0$, hence $b\ge0$. Thus $D_i$ is semisimple with non-negative
real spectrum on every $A_\chi$.

Let $A_b(D_i)=\ker(D_i-b\id)$ and set
\begin{equation}\label{eq:mr-BJ}
 B_i=\ker D_i,\qquad J_i=\bigoplus_{b>0}A_b(D_i),\qquad A=B_i\oplus J_i.
\end{equation}
The derivation rule adds eigenvalues under multiplication. Hence $B_i$ is a subalgebra,
$J_i$ is an ideal, and projection $\epsilon_i:A\to B_i$ is an algebra homomorphism
retracting the inclusion. In particular $B_i\simeq A/J_i$ is finitely generated.
It is normal: if $u\in\operatorname{Frac}(B_i)$ is integral over $B_i$, then $u\in A$ by normality,
and $D_i u=0$, so $u\in B_i$.
We set $Q_i:=\Spec B_i$. From the above discussions, we obtain 
\[
 j_i:Q_i\hookrightarrow C,\qquad \rho_i:C\to Q_i,\qquad \rho_i j_i=\id.
\]
Hence, $Q_i$ can be regarded as an ``algebro-geometric 
retraction" of $C$. 
Choose finitely many algebra generators $f_1,\ldots,f_N$ to embed $C$ in $\C^N$.
Every $f_\nu$ is a finite sum of $D_i$-eigenvectors, so
$f_\nu(\Psi_{i,t}x)\to\epsilon_i(f_\nu)(x)=f_\nu(j_i\rho_i(x))$.
Coordinatewise convergence in this closed embedding therefore gives
\begin{equation}\label{eq:mr-retractionlimit}
 j_i\rho_i(x)=\lim_{t\to-\infty}\Psi_{i,t}(x)\qquad(x\in C).
\end{equation}
For $x\in U$, \eqref{eq:mr-complement-flow} shows that
$r^2(\Psi_{i,t}x)\to h_i(x)$ as $t\to-\infty$.
Apply the comparison with $r_\xi$ along this flow. Since $r_\xi$ is
continuous, \eqref{eq:mr-retractionlimit} gives
\begin{equation}\label{eq:mr-radius}
 \kappa^{-2}r_\xi(j_i\rho_i(x))^2\le h_i(x)
 \le\kappa^2r_\xi(j_i\rho_i(x))^2\qquad(x\in U).
\end{equation}
For a general $x\in U$, we use Lemma \ref{lem:leaves} and consider
the connected maximal leaf of $F_j|_U$ through $x$,
which we denote by $L_j(x)$.

Write $\overline{L_j(x)}$ for its Zariski closure 
with the reduced scheme structure. 

For $i\ne j$, $dh_i|_{F_j}=0$, so $h_i$ is constant on $L_j(x)$.
Identifying each $Q_i$ with its closed image in $C$,
\eqref{eq:mr-radius} gives
\[
 r_\xi(\rho_i(y))\le\kappa\sqrt{h_i(x)}\qquad(y\in L_j(x)).
\]
The irreducible closure $\overline{L_j(x)}$ has a dense smooth open subset tangent to $F_j$ and
contained in this leaf; hence $L_j(x)$ is analytically dense in $\overline{L_j(x)}$.
Continuity of $r_\xi$ and $\rho_i$ extends this bound to the closure.
Properness of $r_\xi$ therefore places $\rho_i(\overline{L_j(x)})$
in a compact subset of $C$. Every affine coordinate of this image
is bounded. A nonconstant regular function on
$\overline{L_j(x)}$ would have constructible dense, hence cofinite, image in $\mathbb A^1$.
Thus $\rho_i|_{\overline{L_j(x)}}=\rho_i(x)$ for $i\ne j$.
Varying the general leaf and using algebraicity gives
\begin{equation}\label{eq:mr-leaves}
 d\rho_i(F_j)=0\quad(i\ne j),\qquad \dim Q_i\leq n_i.
\end{equation}
The differential assertion is understood on pulled-back regular functions.
Since $\rho_i$ is surjective, its generic rank is $\dim Q_i$; the vanishing 
on the complementary summands $\oplus_{j\le i}F_j$ 
gives the dimension bound. 
In particular,
\begin{equation}\label{eq:mr-constants}
 B_i=\{f\in A\mid df(F_j)=0\text{ for all }j\ne i\}.
\end{equation}
One inclusion follows from \eqref{eq:mr-leaves}; for the other, $D_i=\sum_{j\ne i}E_j$
is tangent to those directions.

As the product of the retraction morphism $\rho_i$s, we set 
$$p:=(\rho_1,\ldots,\rho_k):C\to Q=\prod_iQ_i$$ 
and let $B\subset A$ be the image of 
$\bigotimes_i B_i$. Use the original grading and denote the 
ideal with positive $\xi$-weights
of $B$ by $B_+$. For $x\in U$, the comparison with $r_\xi$ and
\eqref{eq:mr-radius} give
\[
 \kappa^{-2}r_\xi(x)^2\le r^2(x)=\sum_i h_i(x)
 \le\kappa^2\sum_i r_\xi(j_i\rho_i(x))^2.
\]
Thus
\begin{equation}\label{eq:mr-joint-radius}
 r_\xi(x)^2\le\kappa^4\sum_i r_\xi(j_i\rho_i(x))^2\qquad(x\in C).
\end{equation}
Indeed, the inequality holds on $U$, and both sides are continuous
on $C$, so density extends it to all of $C$. If
$\rho_i(x)=\rho_i(o)$ for every $i$, then the right-hand side is zero;
hence $r_\xi(x)=0$ and $x=o$. We obtain
$V(B_+A)=p^{-1}(\rho_1(o),\ldots,\rho_k(o))=\{o\},
 \qquad \dim_\C A/B_+A<\infty.$ 
Choose homogeneous lifts of a basis of this quotient. They generate 
$A$ over $B$ for an elementary ring-theoretic reason: 
after subtracting a linear combination of the lifts, a homogeneous element is a sum of 
positive $\xi$-weight elements of $B$ times elements of strictly 
smaller $\xi$-weight. Then we repeat the process. 
Thus $p$ is finite onto its closed image. By \eqref{eq:mr-leaves},
\[
 \dim C=\dim p(C)\leq\sum_i\dim Q_i\leq\sum_i n_i=\dim C, 
\]
so $p$ is surjective and $\dim Q_i=n_i$. 

Give $A$ the Poisson structure induced by $\sigma_C$. By 
\eqref{eq:mr-constants}, the Hamiltonian 
field of an element of $B_i$ lies in $F_i$. Its 
involutivity and the general identity 
$X_{\{f,g\}}=[X_f,X_g]$ show that $B_i$ is a Poisson subalgebra of $A$; moreover $\{B_i,B_j\}=0$ for $i\ne j$. 
Hence $Q$ has the product bracket and $p$ preserves 
the Poisson structure. 
$p$ is generically \'etale, so the Poisson 
bracket on the target $Q$ is also generically nondegenerate.
Let $\eta$ be the top exterior power of its inverse rational 2-form. For a prime divisor
$D'$ of $C$ lying over $D$ of $Q$, put $e=e(D'/D)$ and $a=\operatorname{ord}_D\eta\in\Z$.
The pullback $p^*\eta$ is the symplectic volume form. The standard ramification calculation gives $0=\operatorname{ord}_{D'}p^*\eta=ea+(e-1),\qquad e(a+1)=1$ so that $e=1$. This proves quasi-\'etaleness, and the Zariski-Nagata purity 
(cf., e.g., \cite[Lemma 58.21.4, Tag 0BMB]{Stacks})
implies $p$ is \'etale over $Q^{\reg}$. 

Set $U_i=Q_i^{\reg}$, $U_Q=\prod_iU_i$, and $C^\circ=p^{-1}(U_Q)$.
With basepoints chosen, the connected finite \'etale cover $C^\circ\to U_Q$ corresponds to
a finite-index subgroup $H\subset\Pi=\prod_i\Pi_i$, where $\Pi_i=\pi_1(U_i)$ are ordinary
topological fundamental groups. Regard each $\Pi_i$ as a coordinate subgroup and put
\begin{equation}\label{eq:mr-subgroups}
 H_i=H\cap\Pi_i,\qquad H_0=\prod_iH_i,\qquad \Gamma=H/H_0.
\end{equation}
Each $H_i$ has finite index in $\Pi_i$, and $H_0\triangleleft H$ since $\Pi_i\triangleleft\Pi$.
The Riemann existence theorem \cite[Expos\'e XII, Th\'eor\`eme 5.1]{SGA1}
algebraizes the covers $V_i^\circ\to U_i$ corresponding to $H_i$.
Their product covers $C^\circ$ with Galois group $\Gamma$.
For $h=(h_1,\ldots,h_k)\in H$, each $h_i$ normalizes $H_i$, so this action is factorwise.
If it is trivial on every factor except possibly $i$, then $h_j\in H_j$ for $j\ne i$.
Removing those coordinate elements leaves an element of 
$H\cap\Pi_i=H_i$. 
Thus $h\in H_0$. For any $({\rm id}\neq )g\in \Gamma$, its action 
on the factors 
$\{V_i\}_i$ is nontrivial for more than $i$s.  

Consider the spectra of the normalization of $Q_i$ 
in  $\C(V_i^\circ)$  to obtain a finite normal affine cover $V_i\to 
Q_i$. 
The factorwise action extends to $V=\prod_iV_i$. Both $V/\Gamma$ and $C$ are normal and
finite over $Q$, and agree over $U_Q$; equivalently, they are its normalization in the same
function field. Therefore $V/\Gamma\simeq C$. Since $C\setminus C^\circ$ has codimension
at least two, $\phi:V\to C$ is quasi-\'etale. 

The purity implies that $V^\dagger=\phi^{-1}(C^{\reg})\to C^{\reg}$ 
is an \'etale cover. 
The forms $\phi^*\sigma_i$ extend uniquely from $V^\dagger$ to  
$V^{\reg}$ by the Hartogs extension and their sum clearly remains 
nondegenerate. 
On $\prod_iV_i^\circ$, the \'etaleness of $p$ together with
\eqref{eq:mr-leaves} and equality of dimensions identifies $F_i$
with the $i$-th summand. 
Thus $\phi^*\sigma_i$ is defined on 
$V^{\reg}$ and $\phi^*\sigma_i=\pr_i^*\omega_i$. 

The quasi-\'etale cover $V$ is klt \cite[Proposition 5.20]{KM}; then each $V_i$ is klt by
taking smooth slices in the other factors. The extension theorem
\cite[Theorem 1.4]{GKKP} extends $\omega_i$ regularly to a resolution.
Thus each $V_i$ is symplectic.
A nonidentity finite-order symplectic automorphism has fixed locus of codimension at
least two: at a smooth fixed point its fixed tangent space is symplectic and the action
is locally linearizable, while the singular locus already has codimension at least two.
Every $\gamma\ne1$ moves at least two factors, so
\[
 \operatorname{Fix}_V(\gamma)=\prod_i\operatorname{Fix}_{V_i}(\gamma_i),\qquad
 \codim_V\operatorname{Fix}_V(\gamma)\geq4
\]
when the fixed locus is nonempty. The image of the union of these fixed loci is the
required $Z$; a finite map preserves dimension and the quotient is \'etale off that image.

Finally, we consider the action of 
two tori after replacing by isogenies. 
Put $G=T\times G_{\test}$ and choose a product isogeny 
$\alpha\times\beta:G'=T'\times G'_{\test}\to G$ killing the finite orbit 
monodromy $\pi_1(G)\to\Gamma$ of $V^\dagger\to C^{\reg}$.
The covering-space lifting criterion lifts the action to $V^\dagger$,
normalized to be the identity over $1\in G'$. Uniqueness gives the group law
and commutation with $\Gamma$. Riemann existence algebraizes this lift, and
normality extends it across the codimension-two complement to $G'\times V$.
The two lifted torus actions therefore commute. Since the original actions preserve each $F_i$, their lifts preserve 
each tangent summand $\pr_i^*T_{V_i^{\reg}}$ in 
$T_{V^{\reg}}=\bigoplus_i\pr_i^*T_{V_i^{\reg}}.$ 
So, for every $g'\in G'$, the differential of $\pr_i\circ g'$
vanishes in the directions tangent to $V_j$ whenever $j\ne i$.
Thus, the lifted action induces an action 
on each $V_i$. 

Let us write $A_V=\C[V]$, so $A_V^\Gamma=A$, and lift $\xi$ to
$\xi'=(d\alpha)^{-1}\xi$. 
Below, we show that the $T'$-action is good on $V$ 
by standard arguments: 
for a $\xi'$-homogeneous element $s\in A_V$ of
weight $w$, consider 
\[
 P_s(X):=\prod_{\gamma\in\Gamma}(X-\gamma^*s)
 \in A_V^\Gamma[X]=A[X].
\]
Since $T'$ commutes with $\Gamma$, the coefficient of $X^{|\Gamma|-j}$
has weight $jw$. If $w<0$, all nonleading coefficients vanish and
$P_s(s)=0$ forces $s=0$; if $w=0$, then $P_s\in\C[X]$, so that $s$ is constant. Hence, $\xi'$ is a positive Reeb
vector on $V$ and the desired goodness of the $T'$-action 
also follows. 

Let $\eta_i$ be the $G_{\test}$-character of $\sigma_i$, whose
$T$-character is $\tau$. For $g'=(t',u')$, write
$g=(\alpha(t'),\beta(u'))$ and let $g'_i$ denote the induced action of
$g'$ on $V_i$.
The identity $\pr_i^*\omega_i=\phi^*\sigma_i$ established above and
$\phi\circ g'=g\circ\phi$ give, first on $V^\dagger$ and then on
$V^{\reg}$ by extension,
\begin{align*}
 \pr_i^*((g'_i)^*\omega_i)
 &=(g')^*\phi^*\sigma_i=\phi^*(g^*\sigma_i)\\
 &=\tau(\alpha(t'))\,\eta_i(\beta(u'))\,\pr_i^*\omega_i.
\end{align*}
Restricting to a slice through smooth points of the other factors shows
that both lifted tori preserve $(V_i,\C\omega_i)$, whereas $\Gamma$
preserves $\omega_i$ itself. Since $\langle\alpha^*\tau,\xi'\rangle=2$,
an integral multiple of a rational direction sufficiently close to $\xi'$
is positive on finitely many homogeneous algebra generators and on the form
character. It gives a good action on each factor, making 
$\omega_i$ homogeneous with a positive weight. 
\end{proof}

\begin{remark}
Note that we did {\it not} 
claim that retract $Q_j$ is isomorphic to 
$\overline{L_j(x)}$ for general $x$. 
There is a priori a possibility such that, 
for example, in  
$C=(\C^2\times\C^2)/\langle(-1,-1)\rangle$, the leaf with a fixed
nonzero second vector has closure isomorphic to $\C^2$, whereas
$Q_1\simeq\C^2/\{\pm1\}$. The restriction of $\rho_1$ to that
closure is the degree-two quotient map, not an isomorphism.     
\end{remark}

For the rest of the proof, we fix the symplectic factors $V_i$
and the finite Galois cover $\phi:V=\prod_iV_i\to C$
constructed above. 
We continue with this same cover: \S \ref{sec:extendcover} extends
it over the formal family, and \S \ref{sec:tensors} supplies the
relative tensors and Casimir vanishing needed before applying
Theorem \ref{thm:SPD} in \S \ref{sec:recursive}.

\subsection{Deformation of finite cover}\label{sec:extendcover}
Recall $\Gamma$, the deck group of $V=\prod_iV_i\to C$, and $Z$, the branch locus of codimension at least $4$ 
from Theorem \ref{thm:cover}. The next lemma first extends the  underlying finite cover; a Poisson structure on the deformed cover is added afterwards.
Fix once and for all a product isogeny $\widetilde G=\widetilde T\times\widetilde G_{\test}\to T\times G_{\test}$ that lifts the actions in Theorem \ref{thm:cover}. In Sections \ref{sec:extendcover}-\ref{sec:recursive}, equivariance on the covered family refers to these lifted tori; to lighten notation we continue to write $T$ and $G_{\test}$ there. We keep the original groups downstairs. The finite kernel acts upstairs by deck transformations, and this fact will restore the original actions after taking $\Gamma$-invariants.
\begin{lemma}[Infinitesimal deformation of a finite cover]\label{lem:extendcover}
Let $\phi:V\to C$ be a finite Galois cover 
between normal affine Cohen-Macaulay varieties, \'etale over  $U=C\setminus Z$, with 
\[
\depth_Z\phi_*\OO_V\ge3,\qquad\depth_Z\OO_C\ge3.
\]
For every marked flat deformation $C_\Lambda$ over a local Artin $\C$-algebra $\Lambda$ with residue field $\C$, there is a unique flat deformation $V_\Lambda$ and finite map $V_\Lambda\to C_\Lambda$ extending this cover. Its deck action extends, $V_\Lambda/\Gamma=C_\Lambda$, and the construction is compatible with the given group actions and the base changes of the 
base Artin local ring. 
\end{lemma}
\begin{proof}
Note that $C$ and $C_\Lambda$ have 
the same underlying topological spaces. Thus, the open subset $U\subset C$ can be enhanced to an open subscheme 
$U_\Lambda\subset C_\Lambda$. Over $U_\Lambda$, it 
has a unique finite \'etale algebra $\mathcal E_\Lambda$ lifting  the cover $\phi^{-1}(U)\to U$, by \cite[Tag 0BQB (58.8.3)]{Stacks}. Set  $M_\Lambda:=j_*\mathcal E_\Lambda$. 
For a small extension $\Lambda\twoheadrightarrow\Lambda_0$ with kernel $I$, push forward the exact sequence 
\[
0\to I\otimes_{\C}E\to E_\Lambda\to E_{\Lambda_0}\to 0
\]
on $U_\Lambda$ by $(R^i)j_*$. 
The depth assumption gives $R^1j_*(\phi_*\OO_V|_U)=0$, since it is the local-cohomology sheaf $\mathcal H_Z^2(\phi_*\OO_V)$. 
Also, since $\depth_Z(\phi_*\OO_V)\ge2$, we have
$j_*\bigl((\phi_*\OO_V)|_U\bigr)\simeq\phi_*\OO_V$. 
Thus
\begin{equation}\label{eq:cover-exact}
0\longrightarrow I\otimes_\C\phi_*\OO_V\longrightarrow M_\Lambda\longrightarrow M_{\Lambda_0}\longrightarrow0.
\end{equation}
Induction on the length of $\Lambda$ proves coherence of $M_\Lambda$. Put $F=\phi_*\OO_V$ and $\mathfrak m=\mathfrak m_\Lambda$. Then $\ker(M_\Lambda\to M_{\Lambda_0})=IM_\Lambda$,
\[ M_\Lambda/IM_\Lambda\simeq M_{\Lambda_0},\qquad M_\Lambda/\mathfrak mM_\Lambda\simeq F \]
from the above. 
Thus multiplication $I\otimes_\Lambda M_\Lambda\simeq I\otimes_\C F\to M_\Lambda$ is injective, giving $\operatorname{Tor}_1^\Lambda(\Lambda_0,M_\Lambda)=0$.
The assumption of induction 
($\Lambda_0$-flatness of $M_{\Lambda_0}$) 
and \cite[Lemma 10.101.6]{Stacks} now implies $\Lambda$-flatness of $M_\Lambda$. Thus we obtain a finite flat deformation $V_\Lambda=\Spec_{C_\Lambda}M_\Lambda$. 

Any other flat finite extension 
${\rm Spec}_{C_\Lambda}(N)$ 
over $\Lambda$ satisfies 
$N\simeq j_*(N|_{U_\Lambda})$ 
by $\depth_Z\phi_*\OO_V\ge2$.
Thus we have the uniqueness 
$N\simeq j_*\mathcal E_\Lambda=M_\Lambda$, 
by \cite[Lemma 58.8.3, Tag 0BQB]{Stacks}. 
The existence of the $\Gamma$-action also follows from 
the uniqueness and the compatibility with the base change 
also follows from the construction. 
\end{proof}

Note that we can apply the above lemma \ref{lem:extendcover} 
to the setup of  Theorem \ref{thm:cover} i.e., to its  
$C_n=\Spec(\cR/t^{n+1}\cR)$. The compatible inverse limit is 
a finite formal morphism 
\begin{equation}\label{eq:formalcover}
\tilde{\phi}\colon \widehat\cV\longrightarrow\widehat\cW,\qquad
\RV:=\Gamma(\widehat\cV,\OO),\qquad
\RV/t\RV=\Gamma(V,\OO_V),\qquad(\RV)^\Gamma=\hR.
\end{equation}
The two torus actions and $\Gamma$ naturally lift to 
$\widehat\cV$. 

\subsection{Vanishing of homogeneous Casimir elements}\label{sec:tensors}
Now we work on the obtained finite formal cover $\tilde{\phi}\colon \widehat\cV\longrightarrow\widehat\cW$, and prepare a key 
lemma for the next subsection  \S \ref{sec:recursive}. 
That is, we lift the tensors of 
\S \ref{sec:flag} and use their inverse identity to prove 
a certain Casimir vanishing theorem. 

Pull back $\Pi$ on the \'etale locus of 
$\tilde{\phi}$ in \eqref{eq:formalcover}; \'etale base change uniquely lifts biderivations. The construction by $j_*$ extends its values to $\RV$, 
which thus obtains a regular formal Poisson bracket. 
We pull back $S$ to the covered formal family and use the same
notation for its unique extension to the relative smooth locus.
The pulled-back tensors continue to satisfy
\eqref{eq:inverse}.

For the recursive splitting in \S \ref{sec:recursive}, we 
recall 
the original elementary-divisor exponents and the symplectic
forms on the factors of the central fiber $V=\prod_iV_i$:
\begin{equation}\label{eq:fixed-factors}
0=e_1<\cdots<e_k=Q:=\qmax,\qquad
A_i:=\Gamma(V_i,\OO_{V_i}),\quad
\pi_i:=\sigma_{V_i}^{-1},
\end{equation}
where $\sigma_{V_i}=\omega_i$ is the form constructed in
Theorem \ref{thm:cover}.

On the smooth locus $V^{\reg}$ of the central fiber
$V=\prod_iV_i$, the leading forms are
$\phi^*\sigma_i=\pr_i^*\sigma_{V_i}$.
By the arguments of Proposition \ref{prop:blocks},
applied to their nondegenerate restrictions to the corresponding 
tangent summands, the identity $\Pi=t^QS^{-1}$ 
induces a natural correspondence between 
$t^{e_i}\sigma_{V_i}$ and $t^{Q-e_i}\pi_i.$ 
Consequently, the coordinate algebra and Poisson bracket
on the central fiber $V=\prod_iV_i$ are 
\[
\RV/t\RV=\bigotimes_{i=1}^k A_i,\qquad
\Pi_0=\pi_k+0_{\prod_{i<k}V_i}.
\]

Now, we prepare the following key lemma 
(which partly replaces the role of the argument in 
\cite[\S 3, Lemma 3.11]{NO} in a different manner.) 

\begin{lemma}[Vanishing of homogeneous Casimir elements]\label{lem:casimir}
Let $f\in\RV$ be Casimir for $\Pi$ and satisfy $E(f)=\lambda f$ with $\lambda\ne0$, where $E=Z_\xi$ is the fiberwise Euler field, so $E(t)=0$. Then $f=0$.
\end{lemma}
\begin{proof}
On the relative smooth formal locus above $C^{\reg}$, centrality and \eqref{eq:inverse} give
\begin{equation}\label{eq:casimir-vanishing}
0=\iota_E S^\flat\Pi^\sharp(\drel f)=t^Q E(f)=\lambda t^Qf.
\end{equation}
The formal structure sheaf is $t$-torsion-free by flatness, so $f=0$ on this open set. Restriction from $\RV$ to this open set is injective: a section restricting to zero has zero reduction on a dense open subset of the integral central fiber, hence is divisible by $t$; divide and repeat, then use $t$-adic separatedness. Thus $f=0$ globally.

\end{proof}

\begin{remark}
        More precisely, if $n+1\ge Q$, a homogeneous Casimir of nonzero Reeb weight modulo $t^{n+1}$ is zero modulo $t^{n+1-Q}$. Indeed, at order $n$ the structure sheaf is flat over $\Lambda_n$, not $t$-torsion-free. Flatness gives  $\Ann(t^Q)=t^{n+1-Q}\OO$.
\end{remark}

\subsection{Recursive splitting of the original product factors}\label{sec:recursive}
The proof of the following 
proposition \ref{prop:formal-product} 
is where Theorem \ref{thm:SPD}, which we prepared 
by using Harrison-Poisson cohomology, is applied in the main arguments. 
Let us set 
\[
A_{\le j}:=\bigotimes_{i\le j}A_i,\qquad A_{>j}:=\bigotimes_{i>j}A_i,
\qquad A_{\le0}=A_{>k}=\C.
\]
\begin{proposition}[Formal splitting]\label{prop:formal-product}
There is a marked $T\times G_{\test}\times\Gamma$-equivariant isomorphism of formal commutative algebras
\begin{equation}\label{eq:formal-product}
\RV\simeq\left(\bigotimes_{i=1}^k A_i\right)\cten_\C\C[[t]]
=\Gamma(V,\OO_V)\cten_\C\C[[t]].
\end{equation}
Under this isomorphism, on the formal relative smooth locus,
\begin{equation}\label{eq:tensor-normalform}
S=\bigoplus_{i=1}^k t^{e_i}\sigma_{V_i},\qquad
\Pi=\sum_{i=1}^k t^{Q-e_i}\pi_i.
\end{equation}
The second equality also holds as an equality of Poisson structures on the whole formal algebra. 
\end{proposition}
The proof is a factor-by-factor recursive argument, 
in decreasing order of $e_i$. 
In particular, the first splitting and all later splittings 
are instances of the same step. 
\begin{proof}
We descend through the finite list $j=k,k-1,\ldots,1$. At stage $j$ we keep  $\cB^{(j)}$, a regular relative 2-form $S^{(j)}$, and a regular Poisson bracket $\Pi^{(j)}$, with a marked equivariant identification
\begin{equation}\label{eq:recursive-invariant}
\begin{aligned}
\RV&\simeq A_{>j}\cten_\C\cB^{(j)},
&\cB^{(j)}/t\cB^{(j)}&=A_{\le j},\\
S&=\left(\bigoplus_{i>j}t^{e_i}\sigma_{V_i}\right)\oplus S^{(j)},\\
\Pi&=\sum_{i>j}t^{Q-e_i}\pi_i+t^{Q-e_j}\Pi^{(j)},
&\Pi^{(j)}&=t^{e_j}(S^{(j)})^{-1}.
\end{aligned}
\end{equation}
The flag of $S^{(j)}$ on the smooth locus of the 
central fiber 
consists of the remaining original factors $V_i$ for $i\le j$, with their original orders $e_i$ and leading forms $\sigma_{V_i}$. In particular, 
\begin{equation}\label{eq:recursive-center}
(\cB^{(j)}/t,\Pi^{(j)}_0)
=(A_j\otimes A_{\le j-1},\pi_j+0).
\end{equation}
For $j=k$, all these assertions are the constructions of \S \ref{sec:tensors}, with $\cB^{(k)}=\RV$, $S^{(k)}=S$, and $\Pi^{(k)}=\Pi$.

\emph{Separate the current highest-order factor.}
Apply Theorem \ref{thm:SPD} to \eqref{eq:recursive-center}, with
\[
A=A_j,\qquad B=A_{\le j-1}.
\]
Theorem \ref{prop:namikawa} applies to $A_j$. In particular,
the degree-two cohomological input is
\eqref{eq:HP2}, with the zero bracket on $A_{\le j-1}$:
\begin{equation}\label{eq:recursive-HP2}
 \HP^2(A_j\otimes A_{\le j-1},\pi_j+0)
 \simeq \HP^2(A_j)\otimes A_{\le j-1}
       \oplus\HP^2(A_{\le j-1},0).
\end{equation}
This is the same cited Ginzburg-Kaledin formula from
\S \ref{sec:PD-kunneth}, not a further product theorem.
The mixed summand of the formula for the degree $3$ case \eqref{eq:HP3}, 
applied to the same two factors in Step 2 of Theorem \ref{thm:SPD}, 
makes the resulting parameter values actual Casimir elements.
We obtain $\cB^{(j-1)}$ and a presentation 
\begin{align}\label{decompsfamilles}
\cB^{(j)}
\simeq \cU_{A_j}\cten_{D_{A_j}}\cB^{(j-1)},\qquad
h_\alpha\longmapsto
f_\alpha\in t\cB^{(j-1)}\cap\ZP(\cB^{(j-1)}).
\end{align}
The index $\alpha$ labels the parameters of this particular universal base; it is unrelated to the block index $j$. Choose its coordinates to be $T$-eigenvectors. Equivariance of the base map makes each $f_\alpha$ homogeneous of the same positive Reeb weight.

Since $f_\alpha\in\ZP(\cB^{(j-1)})$ is also the image of
the Poisson-central parameter $h_\alpha\in D_{A_j}$,
we have $\{f_\alpha,-\}_{\Pi^{(j)}}=0$ on $\cB^{(j)}$.
Its brackets with $A_{>j}$ are zero by
\eqref{eq:recursive-invariant}, so
$\{f_\alpha,-\}_{\Pi}=0$ on $\RV$.
Lemma \ref{lem:casimir} therefore gives $f_\alpha=0$. Hence
\begin{equation}\label{eq:one-step-split}
\cB^{(j)}\simeq A_j\cten_\C\cB^{(j-1)},\qquad
\Pi^{(j)}=\pi_j+\beta^{(j-1)}(t),\qquad
\beta^{(j-1)}(0)=0.
\end{equation}

Under the isomorphism in \eqref{eq:one-step-split}, we have 
$S^{(j)}=t^{e_j}(\pi_j+\beta^{(j-1)}(t))^{-1}$, which gives 
\begin{equation}\label{eq:one-step-form}
S^{(j)}=t^{e_j}\sigma_{V_j}\oplus S^{(j-1)},\qquad
S^{(j-1)}=t^{e_j}\bigl(\beta^{(j-1)}(t)\bigr)^{-1}.
\end{equation}
The residual form is regular, since it is the difference of regular forms. Its regularity as a form on the residual factor can be checked on a slice through any fixed smooth point of the constant factor $V_j$. The flag and pairings \eqref{eq:flag}-\eqref{eq:leading} are intrinsic to $S$ and $t$. Since every isomorphism is the identity on the central fiber, they are transported to the same original flag and pairings. Thus $S^{(j-1)}$ has exactly the original remaining orders $e_1,\ldots,e_{j-1}$ and forms.

If $j>1$, its largest order is consequently
\[
e_{j-1}=\max\{e_i\mid i<j\}.
\]
We now get
\begin{equation}\label{eq:one-step-renormalize}
\beta^{(j-1)}(t)=t^{e_j-e_{j-1}}\Pi^{(j-1)},\qquad
\Pi^{(j-1)}=t^{e_{j-1}}(S^{(j-1)})^{-1},
\end{equation}
where $\Pi^{(j-1)}$ is regular and restricts to $\pi_{j-1}+0$ at $t=0$. Indeed, the local normal form shows that every value of
$\beta^{(j-1)}(t)$ is divisible by $t^{e_j-e_{j-1}}$
on the relative smooth locus. This divisibility holds
in $\cB^{(j-1)}$ as well. To see this, let
$r=e_j-e_{j-1}$ and let $f$ be such a bracket value.
Its reduction modulo $t$ vanishes on the dense smooth
locus of the integral central fiber, hence vanishes
everywhere. Thus $f=tf_1$. By flatness, multiplication
by $t$ is injective on the relative smooth locus,
so $f_1$ is divisible there by $t^{r-1}$.
Repeating this argument gives
$f\in t^r\cB^{(j-1)}$.
The quotient therefore defines a regular bracket
$\Pi^{(j-1)}$ on the whole formal algebra.
Since $t$ is automatically Casimir, its Leibniz and Jacobi
identities follow from those of $\beta^{(j-1)}(t)$
and $t$-torsion-freeness. 

This recovers \eqref{eq:recursive-invariant} with $j$ replaced by $j-1$. After the step $j=1$, we have
$\cB^{(0)}/t\cB^{(0)}=\C$, hence
$\cB^{(0)}\simeq\C[[t]]$.
Repeating the argument for $j=k,k-1,\ldots,1$
therefore proves \eqref{eq:formal-product}
and \eqref{eq:tensor-normalform}.
\end{proof}

Thus, we have obtained triviality of the formal deformation 
$\tilde{\mathcal{V}}$ of 
the product cover $V$. 
In \S \ref{sec:conclusion}, we 
turn to discuss $\widehat\cW$. 

\section{Completion of the proofs of the main statements}\label{sec:conclusion}
Given the last subsection, it remains to 
take $\Gamma$-invariants 
to obtain the formal triviality for $\widehat\cW$, 
and then show that this trivialization can be taken  
algebraically. 
This completes rigidity under \ref{hyp:grading}-\ref{hyp:metric}; only
the final corollary returns to the original germ $(X,x)$ and uses its
comparison with $W$.

\begin{proof}[Proof of Theorem \ref{thm:rigidity}]
Proposition \ref{prop:formal-product} gives a $\widetilde G\times\Gamma$-equivariant formal product, with the lifted groups fixed in \S \ref{sec:extendcover}. The kernel of $\widetilde G\to T\times G_{\test}$ acts by deck transformations, on the formal cover as well as its central fiber, by uniqueness in Lemma \ref{lem:extendcover}. Thus it acts trivially on $\Gamma$-invariants. Since $V/\Gamma=C$ and $\widehat{\cV}/\Gamma=\widehat{\cW}$, taking 
$\Gamma$-invariants gives an isomorphism equivariant for the original algebraic 
torus $T\times G_{\test}$ as 
\begin{equation}\label{eq:downstairs-formal}
\hR\simeq\Gamma(C,\OO_C)\cten_\C\C[[t]].
\end{equation}
For each $T$-character $\gamma$, the module $\cR_\gamma$ is finite free over $\C[t]$ by \ref{hyp:grading}; the corresponding module $\Gamma(C,\OO_C)_\gamma[t]$ on the product side also has finite rank. Choose a $G_{\test}$-homogeneous $\C[t]$-basis $e_1,\ldots,e_s$ of either module, with weights $b_1,\ldots,b_s$. Such a basis exists by lifting a homogeneous basis modulo $t$ and using the graded Nakayama lemma: these finitely generated graded modules are bounded below because $d>0$. The resulting surjection is an isomorphism, since both free modules have the same rank. In its completion, a formal section of test weight $b$ can contain $t^\ell e_j$ only if $b_j+\ell d=b$, where $d=\wt_{G_{\test}}(t)>0$. For each $j$ there is at most one such $\ell\ge0$, so the section is polynomial in $t$.

Recall $A=\Gamma(C,\OO_C)$, and denote the isomorphism 
\eqref{eq:downstairs-formal} by $\widehat\Psi$.
For any $f\in\cR$ homogeneous with respect to both
$T$ and $G_{\test}$, equivariance implies that
$\widehat\Psi(f)$ has the same weights, so the preceding
observation gives $\widehat\Psi(f)\in A[t]$.
Every element of $\cR$ is a finite sum of such weight
vectors; hence $\widehat\Psi(\cR)\subset A[t]$.
The same argument applied to the inverse gives
$\widehat\Psi^{-1}(A[t])\subset\cR$.
Thus the two maps restrict to mutually inverse
$\C[t]$-algebra maps and it gives the required
$T\times G_{\test}$-equivariant isomorphism
$\cR\simeq A[t]$. 
\end{proof}

\begin{proof}[Proof of Corollary \ref{cor:hk}]
Proposition \ref{prop:blocks} constructs $\sigma_C$ and the hyperK\"ahler structures for the given metric. Theorem \ref{thm:rigidity} identifies $W$ with $C$. By \cite[Theorem 6.2]{XZ}, the completed local rings of $x\in X$ and the vertex of $W$ are isomorphic; the earlier comparison under the hypotheses of \cite[\S4]{NO} is consistent with this conclusion.

The isomorphism of completed local rings implies 
the isomorphism of the analytic germs 
$(X,x)^{\mathrm{an}}\simeq(C,o)^{\mathrm{an}}$
by Artin approximation 
(cf.\ \cite[Lemma 37.39.2, Tag 0CAV]{Stacks}
and \cite{HironakaRossi, Artin}). 
\end{proof}

\end{document}